\documentclass[reqno]{amsart}

\usepackage{amsmath} 
\usepackage{amssymb} 
\usepackage{amsthm} 
\usepackage{epsfig}
\usepackage[loose]{subfigure}
\usepackage{psfrag}
\usepackage[usenames,dvipsnames]{pstricks}
\usepackage{pst-plot}
\usepackage[colorlinks,linkcolor=blue]{hyperref} 

\newcommand{\COMMENT}[1]{}

\newcommand{\diam}{\mathop{\mathrm{diam}}}
\newcommand{\E}{\mathbb{E}}
\newcommand{\eps}{\varepsilon}
\newcommand{\Extend}{\mathcal{E}}
\newcommand{\one}{\mathbf{1}}
\renewcommand{\P}{\mathbb{P}}
\newcommand{\R}{\mathbb{R}}
\newcommand{\smid}{\,\middle|\,}

\newtheorem{thm}{Theorem}[section]
\newtheorem{prop}[thm]{Proposition}
\newtheorem{lem}[thm]{Lemma}

\theoremstyle{definition}
\newtheorem{alg}[thm]{Algorithm}

\newtheorem{rmk}[thm]{Remark}
\newtheorem{eg}[thm]{Example}

\makeatletter
\@addtoreset{equation}{section}
\@addtoreset{figure}{section}
\@addtoreset{table}{section}
\makeatother

\hypersetup{
	pdftitle = {Optimal uncertainty quantification for legacy data observations of Lipschitz functions},
	pdfauthor = {T. J. Sullivan, M. McKerns, D. Meyer, F. Theil, H. Owhadi \& M. Ortiz},
	pdfsubject = {2010 MSC: 60E15, 62G99, 65C50, 90C26},
	pdfkeywords = {uncertainty quantification, probability inequalities, non-convex optimization, Lipschitz functions, legacy data, point observations},
}

\title[Optimal UQ for legacy data observations of Lipschitz functions]{
	Optimal uncertainty quantification for legacy data observations of Lipschitz functions
}

\author{T.\ J.\ Sullivan}
\address{
	T.\ J.\ Sullivan \\
	Mathematics Institute \\
	University of Warwick \\
	Coventry \\
	CV4 7AL \\
	UK
}
\email{Tim.Sullivan@warwick.ac.uk}
\urladdr{\url{http://www.warwick.ac.uk/staff/Tim.Sullivan}}

\author{M.\ McKerns}
\address{
	M.\ McKerns \\
	Center for Advanced Computing Research \\
	California Institute of Technology \\
	1200 East California Boulevard \\
	Mail Code 158-79 \\
	Pasadena \\
	CA 91125 \\
	USA
}
\email{mmckerns@caltech.edu}
\urladdr{\url{http://www.its.caltech.edu/~mmckerns/}}

\author{D.\ Meyer}
\address{
	D.\ Meyer \\
	Lehrstuhl f{\"u}r Numerische Mechanik \\
	Technische Universit{\"a}t M{\"u}nchen \\
	Boltzmannstrasse 15 \\
	D-85747 \\
	Garching bei M{\"u}nchen \\
	Germany
}
\email{meyer@lnm.mw.tum.de}
\urladdr{\url{http://www.lnm.mw.tum.de/Members/meyer}}

\author{F.\ Theil}
\address{
	F.\ Theil \\
	Mathematics Institute \\
	University of Warwick \\
	Coventry \\
	CV4 7AL \\
	UK
}
\email{f.theil@warwick.ac.uk}
\urladdr{\url{http://www.maths.warwick.ac.uk/~theil/}}

\author{H.\ Owhadi}
\address{
	H.\ Owhadi \\
	Applied \& Computational Mathematics and Control \& Dynamical Systems \\
	California Institute of Technology \\
	Mail Code 9-94 \\
	1200 East California Boulevard \\
	Pasadena \\
	CA 91125 \\
	USA
}
\email{owhadi@caltech.edu}
\urladdr{\url{http://www.acm.caltech.edu/~owhadi/}}

\author{M.\ Ortiz}
\address{
	M.\ Ortiz \\
	Division of Engineering and Applied Science \\
	California Institute of Technology \\
	Mail Code 105-50 \\
	1200 East California Boulevard \\
	Pasadena \\
	CA 91125 \\
	USA
}
\email{ortiz@caltech.edu}
\urladdr{\url{http://www.aero.caltech.edu/~ortiz/}}

\date{\today}



\begin{document}

\begin{abstract}
	We consider the problem of providing optimal uncertainty quantification (UQ) --- and hence rigorous certification --- for partially-observed functions.  We present a UQ framework within which the observations may be small or large in number, and need not carry information about the probability distribution of the system in operation.  The UQ objectives are posed as optimization problems, the solutions of which are optimal bounds on the quantities of interest;  we consider two typical settings, namely parameter sensitivities (McDiarmid diameters) and output deviation (or failure) probabilities.  The solutions of these optimization problems depend non-trivially (even non-monotonically and discontinuously) upon the specified legacy data.  Furthermore, the extreme values are often determined by only a few members of the data set;  in our principal physically-motivated example, the bounds are determined by just 2 out of 32 data points, and the remainder carry no information and could be neglected without changing the final answer.  We propose an analogue of the simplex algorithm from linear programming that uses these observations to offer efficient and rigorous UQ for high-dimensional systems with high-cardinality legacy data.  These findings suggest natural methods for selecting optimal (maximally informative) next experiments.
\end{abstract}

\maketitle

\section{Introduction and Outline}

\subsection{Introduction}

In many settings --- including the physical sciences, engineering, and finance --- it is necessary to have a rigorous and also sharp/optimal quantitative understanding of the effects of uncertainties, which are often probabilistic in nature.  Often, the available information about the system of interest comes in the form of \emph{legacy data}, \emph{i.e.}\ a data set that is provided ``as is'' and cannot be extended;  the reasons for such restrictions may range from financial or practical difficulties to legal and ethical concerns.  Uncertainty quantification (UQ) methods for addressing such problems must cope with this non-extensibility, the fact that the distribution of the legacy data may be unrelated to the probability distribution of the system in operation, and that the data set may be either very sparse or very large compared to the system's domain of operation.  This paper approaches the UQ-with-legacy-data problem using the \emph{Optimal UQ} framework proposed in \cite{OwhadiScovelSullivanMcKernsOrtiz:2010}, and thereby develops and illustrates that general framework in a specific setting.

In the Optimal UQ framework \cite{OwhadiScovelSullivanMcKernsOrtiz:2010}, UQ in the presence of both epistemic and aleatoric uncertainties \cite{Limbourg:2005, OberkampfHeltonJoslynWojtkiewiczFerson:2004, RoyOberkampf:2010} is posed as an optimization problem over all feasible scenarios that are consistent with the available information about the input uncertainties --- those uncertainties may be infinite-dimensional in nature, and concern unknown or partially-known probability distributions and functions.  In many cases, the corresponding infinite-dimensional optimization problem can be reduced to an equivalent finite-dimensional problem that allows for closed-form or numerical evaluation \cite[\S3]{OwhadiScovelSullivanMcKernsOrtiz:2010}.

Many UQ methods are not directly applicable if the available data are of legacy type.  For example, in \cite{LucasOwhadiOrtiz:2008}, it was proposed that rigorous certification of physical systems be performed using a concentration-of-measure inequality known as \emph{McDiarmid's inequality} \cite{McDiarmid:1989, McDiarmid:1997, McDiarmid:1998}, also known as the \emph{bounded differences inequality}.  However, this method and its variants \cite{ALLMMOORSS:2011, KLLMOORSS:2011, SullivanTopcuMcKernsOwhadi:2011} require extensive data ``on demand'' in order to compute the McDiarmid diameter, which measures the system output variability and provides the concentration rate in McDiarmid's inequality.  Section \ref{sec:opt_diam} of the present paper shows how the McDiarmid diameter of a Lipschitz function can be optimally bounded using legacy data observations of that function and (upper bounds on) the Lipschitz constants.  

Relationships between the smoothness properties of a function $f$ and bounds on deviation probabilities for $f$ have been studied extensively.  For Lipschitz functions, Talagrand's inequality \cite{Talagrand:1995} is a famous result in this area;  a discussion of non-Lipschitz functions can be found in \cite{Vu:2002}.  However, while such results do use the smoothness information, they do not use arbitrarily-located known values, \emph{i.e.}\ point observations, of $f$.  On the other hand, there are methods that use smoothness information and point observations to calculate the extreme values of $f$ (notably, the algorithm of \cite{JonesPerttunenStuckman:1993} does so without \emph{a priori} knowledge of the Lipschitz constant), but these methods (a) direct further function evaluations, which are not permitted in the context of legacy data, and (b) do not appear to have been coupled to concentration-of-measure methods to produce probability-of-deviation inequalities.  This last point is not surprising, since it is difficult to \emph{prove} a general theorem that will make optimal or near-optimal use of data in advance of knowing those data.

Motivated by this, Section \ref{sec:opt_prob} shows how to \emph{calculate} optimal upper bounds on the probability of deviations from the mean (or any linear function of the system's \emph{a priori} unknown probability distribution) given the legacy data and (upper bounds on) the Lipschitz constants;  this second approach forms part of a large and growing body of work concerning the calculation of optimal inequalities in probability theory --- see \emph{e.g.}\ \cite{BabuskaNobileTempone:2007, BertsimasPopescu:2005, OwhadiScovelSullivanMcKernsOrtiz:2010} for some surveys and historical remarks on this topic.  We find that the extremizers for our optimization problems tend to have a very simple, low-dimensional, singular structure.  Furthermore, once this singular structure has been observed, even approximately, it can be exploited to greatly reduce the computational burden;  see Remark \ref{rmk:dimensional_collapse} and Figure \ref{fig:num_results_cvgce_compare}.

It is also shown that, in certain cases, additional information (in the form of new observations) may not propagate to the resulting bounds, or, dually, that the bounds may be determined by a relatively small ``active'' subset of a large data set.  In Algorithm \ref{alg:legacy_ouq_simplex} we propose an  analogue of the simplex algorithm in linear programming that uses these observations to offer efficient and rigorous UQ for high-dimensional systems with high-cardinality legacy data.  The motivating idea for this algorithm is to solve easier (less constrained) optimization problems when possible, and that the algorithm should terminate in a number of iterations of the same order as the number of relevant data points.  In addition, in the case that the data set can be extended, the optimization formulation of the UQ objectives provides a natural notion of best next experiment (and hence maximally informative data set):  it is the experiment that would induce the greatest change in the extreme value of the UQ optimization problem.

The methods and results of this paper are predicated upon having suitable information (or making assumptions) about the system of interest.  As noted by Hoeffding \cite{Hoeffding:1956}, assumptions about the system of interest play a central and sensitive role in any statistical decision problem, even though the assumptions are often only approximations of reality.  To illustrate the effect of information/assumptions, consider the following toy problem, which will be considered in further detail in Example \ref{eg:Phat_1dim_example} and treated numerically in Subsection \ref{subsec:numerics_1d}:

\begin{eg}
	Suppose that a measurable function $G \colon [0, 1] \to \R$ is applied to a random variable $X$ with unknown distribution on $[0, 1]$, and the event $[G(X) \leq 0]$ is considered to constitute ``failure''.  Given the values of $G$ on some proper (usually finite) subset $\mathcal{O} \subsetneq [0, 1]$, what is the optimal (\emph{i.e.}\ least) upper bound $\widehat{P}$ on the failure probability $\P[G(X) \leq 0]$?  (Note well that the points of $\mathcal{O}$ may be unrelated to the distribution of $X$, and so classical methods of statistical reasoning using the sample set $\{ G(z) \mid z \in \mathcal{O} \}$ are inapplicable.)  
	
	With this information alone, the only rigorous upper bound that can be given is the trivial one:  $\P[G(X) \leq 0] \leq \widehat{P} = 1$.  Consider now the impact of two further pieces of information:
	\begin{enumerate}
		\renewcommand{\labelenumi}{(\Roman{enumi})}
		\item \label{info1} $G$ is Lipschitz continuous with Lipschitz constant $1$, or \emph{short}, \emph{i.e.}
		\[
			| G(x) - G(x') | \leq | x - x' | \text{ for all $x, x' \in [0, 1]$,}
		\]
		and hence $G$ is continuous on $[0, 1]$, and by Rademacher's theorem is differentiable with $| G'(x) | \leq 1$ for Lebesgue-almost-every $x \in [0, 1]$;
		\item \label{info2} some information about the distribution of $X$ on $[0, 1]$ or the distribution of $G(X)$ on $\R$, \emph{e.g.}\ that $\E[G(X)] \geq m$ for some known $m$.
	\end{enumerate}

	\begin{figure}
		\subfigure[Surface plot.]{
			\psfrag{z}{{\scriptsize $z$}}
			\psfrag{G}{{\scriptsize $G(z)$}}
			\psfrag{P}[r]{{\scriptsize $\widehat{P}$}}
			\includegraphics[width=0.4\linewidth]{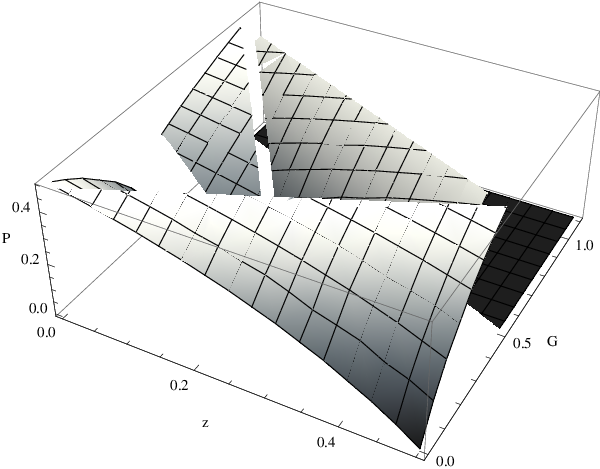}
		}
		\subfigure[Contour plot.]{
			\psfrag{z}{{\scriptsize $z$}}
			\psfrag{G}{{\scriptsize $G(z)$}}
			\includegraphics[width=0.4\linewidth]{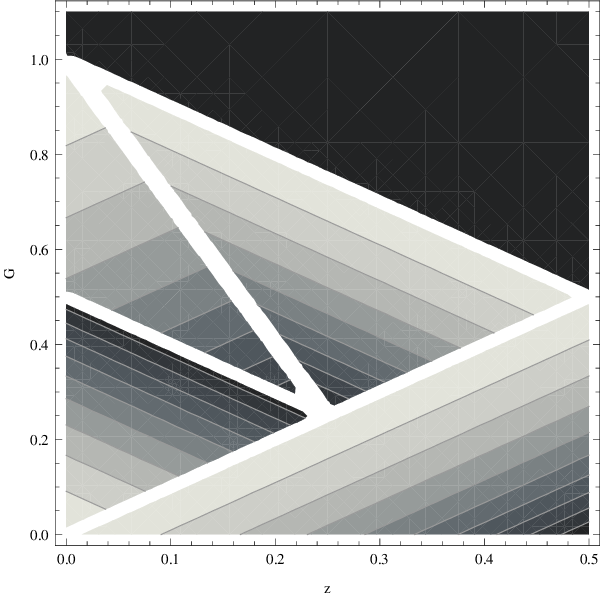}
		}
		\caption{Plots of $\widehat{P}$, the least upper bound on $\P[G(X) \leq 0]$ given that $G \colon [0, 1] \to \R$ has Lipschitz constant $1$, mean $\frac{1}{2}$, and has $(z, G(z))$ on its graph, as a function of $(z, G(z)) \in [0, \frac{1}{2}] \times \R$.  Note the discontinuity and non-monotonicity of $\widehat{P}$ as a function of $(z, G(z))$.}
		\label{fig:discty-intro}
	\end{figure}

	The first item of information does not generally provide any improvement on the trivial upper bound, since although it constrains the set of points $x \in [0, 1]$ for which it is possible that $G(x) \leq 0$, it says nothing about the $\P$-measure of that set, unless it is found to be empty.	However, taken together, $G|_{\mathcal{O}}$ and the two additional items of information \emph{do} provide non-trivial bounds on $\P[G(X) \leq 0]$.  Evaluating these bounds is an infinite-dimensional but well-posed optimization problem, which can be reduced to an equivalent finite-dimensional problem by the reduction theorems of \cite{OwhadiScovelSullivanMcKernsOrtiz:2010}.  Indeed, as will be shown later, if $\mathcal{O}$ consists of one point --- \emph{i.e.}\ we know one point $(z, G(z))$ that lies on the graph of $G$ --- then the least upper bound on $\P[G(X) \leq 0]$ can be given in closed form.  This bound is given in \eqref{eq:discontinuity}, and surface and contour plots are given in Figure \ref{fig:discty-intro}.  Notably, the bound is both non-monotone and discontinuous with respect to the data point $(z, G(z))$.
\end{eg}

\subsection{Outline}

Section \ref{sec:review_and_notation} establishes the notation and set-up of the problems of interest, and recalls a theorem of McShane \cite{McShane:1934} that will be useful later on.

Section \ref{sec:opt_diam} treats the determination of optimal upper bounds on McDiarmid diameters (\emph{i.e.}\ $L^{\infty}$ semi-norms on component-wise oscillations of a function of several independent inputs) using legacy data and Lipschitz constants.  Such upper bounds can be used, together with McDiarmid's inequality and the mean performance of the system, to provide rigorous upper bounds on the system's probability of failure.

Section \ref{sec:opt_prob} treats the problem of directly and optimally bounding the probability of failure, \emph{i.e.}\ finding the least upper bound that is consistent with the legacy data, the Lipschitz constants, and the specified mean performance.  This problem is harder to solve than the problem of Section \ref{sec:opt_diam}, but is still tractable, and has the advantage that it provides the optimal bound on the probability of failure given all the available information, whereas McDiarmid's inequality is non-optimal.

Section \ref{sec:redundancy} discusses necessary and sufficient conditions for a data point to be relevant to the solution of the problems in Sections \ref{sec:opt_diam} and \ref{sec:opt_prob}; put another way, this section concerns the identification of redundant information.

Section \ref{sec:further_remarks} contains some general remarks applicable to both Sections \ref{sec:opt_diam} and \ref{sec:opt_prob}.  

Section \ref{sec:numerics} gives the results of some example numerical implementations of the problems of Section \ref{sec:opt_prob}.  In this section, we see that many data points may be redundant in the sense of Section \ref{sec:redundancy}, and hence that optimal UQ for systems with large legacy data sets may be given by considering well-chosen small subsets of the larger data set.

Section \ref{sec:generalizations} outlines some directions for generalization and future work.

\section{Review and Notation}
\label{sec:review_and_notation}

\subsection{Notation}

Let $(\mathcal{X}_{k}, d_{k})$ be a metric space for each $k \in \{ 1, \dots, K \}$;  prototypically, $\mathcal{X}_{k} = \R$ or $[a_{k}, b_{k}] \subseteq \R$ with the Euclidean distance $d_{k}(x, y) := | x - y |$.  Let $\mathcal{X} := \mathcal{X}_{1} \times \dots \times \mathcal{X}_{K}$.  Let $G \colon \mathcal{X} \to \R$ be some function and suppose that, for each $k \in \{ 1, \dots, K \}$, $L_{k}$ is a global Lipschitz constant for $G$ with respect to its $k^{\text{th}}$ argument:  \emph{i.e.},
\begin{equation}
	\label{eq:Lip_k}
	( x, x' \in \mathcal{X}, x^{j} = x'^{j} \text{ for } j \neq k ) \implies | G(x) - G(x') | \leq L_{k} d_{k}(x^{k}, x'^{k}).
\end{equation}
Define a quasi-metric $d_{L} \colon \mathcal{X} \times \mathcal{X} \to \R$ by
\begin{equation}
	\label{eq:d_Lip}
	d_{L}(x, x') := \sum_{k = 1}^{K} L_{k} d_{k}(x^{k}, x'^{k}).
\end{equation}
If all $L_{k}$ are strictly positive, then $d_{L}$ is a metric.  In the prototypical case, $d_{L}$ is a rescaling of the $\ell^{1}$ ``Manhattan'' metric on $\R^{K}$.

\begin{lem}
	\label{lem:Lipschitz_short}
	A function $f \colon \mathcal{X} \to \R$ is Lipschitz with Lipschitz constant $L_{k}$ in its $k^{\text{th}}$ argument if, and only if, it is short with respect to the metric $d_{L}$, \emph{i.e.}
	\begin{equation}
		\label{eq:Lip_all}
		| f(x) - f(x') | \leq d_{L}(x, x') \text{ for all $x, x' \in \mathcal{X}$.}
	\end{equation}
\end{lem}

\begin{proof}
	Suppose that $f$ is short with respect to $d_{L}$, and let $k \in \{ 1, \dots, K \}$.  Let $x, x' \in \mathcal{X}$ differ only in their $k^{\text{th}}$ component.  Then
	\[
		| f(x) - f(x') | \leq \sum_{j = 1}^{K} L_{j} d_{j}(x^{j}, x'^{j}) = L_{k} d_{k}(x^{k}, x'^{k}),
	\]
	and so $f$ is Lipschitz with Lipschitz constant $L_{k}$ in its $k^{\text{th}}$ argument.  Conversely, suppose that $f$ is Lipschitz with Lipschitz constant $L_{k}$ in its $k^{\text{th}}$ argument, and let $x, x' \in \mathcal{X}$.  Then
	\begin{align*}
		| f(x) - f(x') |
		& \leq | f(x) - f(x'^{1}, x^{2}, \dots, x^{K}) | \\
		& \quad + | f(x'^{1}, x^{2}, \dots, x^{K}) - f(x'^{1}, x'^{2}, x^{3}, \dots, x^{K}) | \\
		& \quad + \dots + | f(x'^{1}, \dots, x'^{K - 1}, x^{K}) - f(x') | \\
		& \leq L_{1} | x^{1} - x'^{1} | + \dots + L_{K} | x^{K} - x'^{K} | \\
		& = d_{L}(x, x'),
	\end{align*}
	and so $f$ is short with respect to $d_{L}$.
\end{proof}

$\mathcal{P}(\mathcal{X})$ denotes the set of all Borel probability measures on $\mathcal{X}$.  The product of probability measures $\mu_{k} \in \mathcal{P}(\mathcal{X}_{k})$ for $k \in \{ 1, \dots, K \}$ will be denoted $\mu_{1} \otimes \dots \otimes \mu_{K}$ or $\bigotimes_{k = 1}^{K} \mu_{k}$;  the set of all such measures will be denoted $\bigotimes_{k = 1}^{K} \mathcal{P}(\mathcal{X}_{k})$.  Recall that if $X = (X_{1}, \dots, X_{K})$ is an $\mathcal{X}$-valued random variable with law $\mu$, then saying that the $K$ components of $X$ are independent is the same as saying that $\mu$ is a product measure $\bigotimes_{k = 1}^{K} \mu_{k}$, where $\mu_{k}$ is the law of $X_{k}$ on $\mathcal{X}_{k}$.

For $f \colon \mathcal{X} \to \R$, let $\mathcal{D}_{k}[f]$ be the \emph{$k^{\text{th}}$ McDiarmid subdiameter} of $f$ on $\mathcal{X}$:
\begin{equation}
	\label{eq:McDiarmid_subdiam}
	\mathcal{D}_{k}[f] := \sup \left\{ | f(x) - f(x') | \smid x, x' \in \mathcal{X}, x^{j} = x'^{j} \text{ for } j \neq k \right\}.
\end{equation}
$\mathcal{D}_{k}[f]$ is a global sensitivity index that measures the sensitivity of $f$ to changes in its $k^{\text{th}}$ argument.  The \emph{McDiarmid diameter} $\mathcal{D}[f]$ of $f$ on $\mathcal{X}$ is defined by
\begin{equation}
	\label{eq:McDiarmid_diam}
	\mathcal{D}[f] := \left( \sum_{k = 1}^{K} \mathcal{D}_{k}[f]^{2} \right)^{1/2}.
\end{equation}
Each $\mathcal{D}_{k}[\cdot]$ (and, indeed, $\mathcal{D}[\cdot]$) is a semi-norm on the space of bounded real-valued functions on $\mathcal{X}$;  any $f$ that is constant in its $k^{\text{th}}$ argument has $\mathcal{D}_{k}[f] = 0$.  

Clearly, if $f \colon \mathcal{X} \to \R$ is known to be $d_{L}$-short, then this information provides a (not necessarily sharp) upper bound on the McDiarmid diameter of $f$:
\begin{equation}
	\label{eq:McDiarmid_Lipschitz}
	\mathcal{D}_{k}[f] \leq L_{k} \diam(\mathcal{X}_{k}, d_{k}) := L_{k} \sup_{x^{k}, x'^{k} \in \mathcal{X}_{k}} d_{k}(x^{k}, x'^{k}).
\end{equation}

The McDiarmid diameter is useful because it places an upper bound on deviations of $f(X)$ from its mean value whenever $X$ is an $\mathcal{X}$-valued random variable with independent components, as the following result shows:

\begin{thm}[McDiarmid's inequality \cite{McDiarmid:1989, McDiarmid:1997, McDiarmid:1998}]
	\label{thm:McDiarmid}
	Let $(\Omega, \mathcal{F}, \P)$ be a probability space and, for $k \in \{ 1, \dots, K \}$, let $X_{k} \colon \Omega \to \mathcal{X}_{k}$ be independent random variables.  Suppose that $\E[|f(X)|]$ is finite.  Then, for any $r > 0$,
	\begin{align}
		\label{eq:McDiarmid_geq}
		\P [f(X) - \E[f(X)] \geq r] & \leq \exp \left( - \frac{2 r^{2}}{\mathcal{D}[f]^{2}} \right), \\
		\label{eq:McDiarmid_leq}
		\P [f(X) - \E[f(X)] \leq -r] & \leq \exp \left( - \frac{2 r^{2}}{\mathcal{D}[f]^{2}} \right).
	\end{align}
\end{thm}

The independence assumption in McDiarmid's inequality can be relaxed and replaced with some control on the martingale differences $\E[f(X)|\mathcal{F}_{i + 1}] - \E[f(X)|\mathcal{F}_{i}]$ of $f(X)$ with respect to a suitable filtration $\mathcal{F}_{\bullet}$ of the probability space $(\Omega, \mathcal{F}, \P)$.  Also, the mean and McDiarmid subdiameters can be used as inputs for sharper inequalities such as the \emph{optimal McDiarmid inequality} of \cite{OwhadiScovelSullivanMcKernsOrtiz:2010}.

Given a measurable system of interest $G \colon \mathcal{X} \to \R$, let $\theta \in \R$ denote a possible value of $G$ that is considered to be a \emph{failure threshold}:  the event $[G(X) \leq \theta]$ represents the failure of the system $G$, and the complementary event $[G(X) > \theta]$ represents the success of the system $G$.  Under the assumption that the random inputs of $G$ (\emph{i.e.}\ the coordinate processes $X_{1}, \dots, X_{K}$) are independent, McDiarmid's inequality implies that the probability of failure is bounded as follows:
\begin{equation}
	\label{eq:McDiarmid_fleqtheta}
	\P [G(X) \leq \theta] \leq \exp \left( - \frac{2 (\E[G(X)] - \theta)_{+}^{2}}{\mathcal{D}[G]^{2}} \right),
\end{equation}
where, for $t \in \R$, $t_{+} := \max \{ 0, t \}$.  The inequality \eqref{eq:McDiarmid_fleqtheta} can be rearranged in order to provide rigorous certification criteria for computational and physical systems of interest, subject to the determination of the mean system performance $\E[G(X)]$ and the McDiarmid diameter $\mathcal{D}[G]$;  see \emph{e.g.}\ \cite{LucasOwhadiOrtiz:2008, KLLMOORSS:2011, ALLMMOORSS:2011}.  Namely, if $p_{\ast} \in [0, 1]$ is the greatest probability of failure that can be accepted if the system is to be called safe, and it is known that $\E[G(X)] \geq m$ and $\mathcal{D}[G] \leq \widehat{D}$, then a sufficient condition for the safety of the system is the truth of the inequality
\begin{equation}
	\label{eq:cert_criterion}
	\frac{(m - \theta)_{+}}{\widehat{D}} \geq \sqrt{\log \sqrt{1 / p_{\ast}}}.
\end{equation}

\subsection{UQ Problem Formulation}

Suppose that the values of $G$ are known only on some \emph{observation set} $\mathcal{O} \subseteq \mathcal{X}$;  that is, the restriction $G|_{\mathcal{O}}$ of $G$ to $\mathcal{O}$ is known exactly.  In applications, it is usually the case that $\mathcal{O}$ is a finite collection of points $\mathcal{O} = \{ z_{1}, \dots, z_{N} \} \subseteq \mathcal{X}$.  Suppose also that constants $L_{1}, \dots, L_{K} \geq 0$ are given such that $G$ is known to be $d_{L}$-short.  The main questions that this paper addresses are the following:  
\begin{enumerate}
	\item Section \ref{sec:opt_diam} shows how to use the observations $G|_{\mathcal{O}}$ and the Lipschitz constants $L = (L_{1}, \dots, L_{K})$ to provide an optimal (\emph{i.e.}\ least) upper bound $\widehat{D}$ on the McDiarmid diameter $\mathcal{D}[G]$.
	\item Section \ref{sec:opt_prob} shows how to use the data $G|_{\mathcal{O}}$, the constants $L$ and the mean performance $\E[G(X)]$ to provide an optimal (\emph{i.e.}\ least) upper bound $\widehat{P}$ on the probability of failure $\P[G(X) \leq \theta]$.
	\item Section \ref{sec:redundancy} considers the problem of determining which observations $z \in \mathcal{O}$ are relevant to the solutions of the problems in the previous two sections.  Furthermore, one can consider the dual problem:  if the data set $G|_{\mathcal{O}}$ could be extended, at what points of the input parameter space $\mathcal{X}$ should $G$ be evaluated to gain maximally relevant information that will improve the bounds $\widehat{D}$ and $\widehat{P}$?
\end{enumerate}

\begin{rmk}[Other UQ problems]
	Although the exposition of this paper treats the \emph{certification} problem of bounding $\P[G(X) \leq \theta]$, there are many other uncertainty quantification problems --- \emph{e.g.} verification, validation, and prediction \cite{OberkampfTrucanoHirsch:2004} --- to which this paper's methods are applicable.  For example, $G$ above may actually stand for the difference between some physical system, $H$, and a model for that system, $F$;  if the aim is to predict values of $H$ using the simulation $F$ with quantified error bounds, then this is tantamount to showing that $\P[ \| H(X) - F(X) \| \geq \theta]$ is suitably small, where $\| \cdot \|$ is some ``error norm'' on (a subset of) parameter space $\mathcal{X}$.  This certification-centric point of view is similar to that of \cite{BarlowProschan:1996}, in which many reliability problems are placed in a unified framework, and that of \cite{OwhadiScovelSullivanMcKernsOrtiz:2010}.
	
	In a different direction to the one taken in this paper, there are important questions of how to make optimal use of legacy data in the calibration and testing of models;  for this problem, a particular difficulty is making best use of the data without \emph{over-fitting} to the data \cite{MorrisonBryantTerejanuMikiPrudhomme:2011}.  The Bayesian perspective is a popular one in this area, and is receiving renewed attention in the context of Bayesian analysis for inverse problems on function spaces \cite{Stuart:2010}.
	
\end{rmk}

\begin{rmk}[Other regularity conditions]
	In many practical applications, of course, the response function $G$ is not known to be globally Lipschitz.  In this paper attention is confined to the globally Lipschitz case as a representative example of a broad class of possible constraints.  For example, it may be more appropriate to consider a H{\"o}lder-type constraint, which would correspond to an inequality of the form
	\begin{equation}
		\label{eq:Hoelder}
		| G(x) - G(x') | \leq d_{L}(x, x')^{\alpha} \text{;}
	\end{equation}
	or a local Lipschitz constraint, which would correspond to an inequality of the form
	\begin{equation}
		\label{eq:locally_Lip}
		| G(x) - G(x') | \leq	\begin{cases}
			d_{L}(x, x'), & \text{if $d_{L}(x, x') < R$,} \\
			+ \infty, & \text{otherwise.}
		\end{cases}
	\end{equation}
	The example of Subsection \ref{subsec:numerics_PSAAP} will use just such a modified Lipschitz constraint, one suited to possibly discontinuous or multivalued functions.  The minimum requirement on any proposed system of inequalities to constrain $G$ is that the desired inequalities should hold whenever $x$ and $x'$ are elements of the observation set $\mathcal{O}$, and that the inequalities must constrain the values of $G$ pointwise.  So, for example, a constraint on the Sobolev $W^{k, p}(\R^{d})$ norm of a function $G \colon \R^{d} \to \R$ would not be a suitable constraint if $k p < d$.  It must be emphasized, though, that if no regularity assumptions are made, then no significant conclusions can be drawn from the data:  regularity is essential if function values at finitely many isolated points are to be used to infer anything about function values elsewhere.
\end{rmk}

\begin{rmk}[Other types of observation]
	In this paper the observations of $G$ are pointwise evaluations of $G$ at finitely many points of its domain.  One could also consider more general observation operators, \emph{e.g.}\ a continuous linear functional $\Lambda \colon W^{k, p}(\R^{d}) \to \R$, or a collection of such operators.
\end{rmk}

\subsection{Extension of Partially-Defined Functions}

In what follows, in order to show that the upper bounds that are obtained are in fact the optimal upper bounds given the available information, it will be necessary to invoke the following extension theorem from metric space theory, which states that a real-valued Lipschitz function defined on any subset of a metric space can always be extended to the whole space without increasing the Lipschitz constant:

\begin{thm}[McShane's extension theorem \cite{McShane:1934}]
	\label{thm:McShane}
	Let $(\mathcal{M}, \rho)$ be a metric space, let $E \subseteq \mathcal{M}$,	and let $C \geq 0$.  If $f \colon E \to \R$ satisfies
	\[
		| f(x) - f(x') | \leq C \rho(x, x') \text{ for all } x, x' \in E,
	\]
	then there exists $\bar{f} \colon \mathcal{M} \to \R$ such that $\bar{f}|_{E} = f$ and
	\[
		\left| \bar{f}(x) - \bar{f}(x') \right| \leq C \rho(x, x') \text{ for all } x, x' \in \mathcal{M}.
	\]
\end{thm}

McShane's theorem also applies to the extension of H{\"o}lder continuous real-valued functions defined on a subset of a metric space;  any continuous real-valued function with concave modulus of continuity can be extended to the whole space while preserving the modulus of continuity.

In the language of metric space theory, McShane's extension theorem says that the Euclidean line $(\R, | \cdot |)$ is an \emph{injective metric space} \cite{Isbell:1964}.  The extension of \emph{vector}-valued Lipschitz functions is a subtle topic:  see \emph{e.g.}\ the Kirszbraun--Valentine theorem \cite{Kirszbraun:1934, Valentine:1945}, which states that Lipschitz functions between Hilbert spaces can always be extended without increasing the Lipschitz constant, which is not generally true even for Lipschitz functions between finite-dimensional Banach spaces \cite[p.\ 202]{Federer:1969}.  It is for this reason that this paper considers only scalar-valued performance measures $G$.

\section{Optimal Bounds on McDiarmid Diameters}
\label{sec:opt_diam}

For each $k \in \{ 1, \dots, K \}$, an upper bound on the McDiarmid subdiameter $\mathcal{D}_{k}[G]$ can be obtained by an optimization problem.  First observe that $\mathcal{D}_{k}[G]$ is the maximum value of the function
\[
	\mathop{\mathrm{diff}_{k}} G \colon \mathcal{X}_{1} \times \dots \times \mathcal{X}_{k} \times \mathcal{X}_{k} \times \dots \times \mathcal{X}_{K} \to \R
\]
defined by
\begin{align*}
	& (\mathop{\mathrm{diff}_{k}} G) (x^{1}, \dots, x^{k - 1}, x^{k}, x'^{k}, x^{k + 1}, \dots, x^{K}) \\
	& \quad := G(x^{1}, \dots, x^{k - 1}, x^{k}, x^{k + 1}, \dots, x^{K}) - G(x^{1}, \dots, x^{k - 1}, x'^{k}, x^{k + 1}, \dots, x^{K}).
\end{align*}
(Indeed, $\mathcal{D}_{k}[G]$ is also the negative of the minimum value of $\mathop{\mathrm{diff}_{k}} G$.)  Therefore, an upper bound on $\mathcal{D}_{k}[G]$ consistent with the observations $G|_{\mathcal{O}}$ and the Lipschitz constant $L = (L_{1}, \dots, L_{K})$ is given by the solution of the following optimization problem in the $K + 3$ variables $x^{1}, \dots, x^{K}, x'^{k}, y, y'$:
\begin{equation}
	\label{eq:opt_diam_1}
	\begin{cases}
		\text{maximize: } & | y - y' | \text{;} \\
		\text{among: } 
		& (x, y) \in \mathcal{X} \times \R \text{,} \\
		& (x', y') \in \mathcal{X} \times \R \text{;} \\
		\text{subject to: }
		& x^{i} = x'^{i} \text{ for all } i \in \{ 1, \dots, K \} \setminus \{ k \} \text{:} \\
		& \quad | y - y' | \leq L_{k} d_{k}(x^{k}, x'^{k}) \text{;} \\
		& \text{for all } z \in \mathcal{O} \text{:} \\
		& \quad | y - G(z) | \leq d_{L}(x, z) \text{,} \\
		& \quad | y' - G(z) | \leq d_{L}(x', z) \text{.}
	\end{cases}
\end{equation}
Note that \eqref{eq:opt_diam_1} is not a linear programming problem:  the feasible set for $(x, y)$ and $(x', y')$ is an intersection of double cones in $\mathcal{X} \times \R$, as illustrated in Figure \ref{fig:feasible}.  Note also that \eqref{eq:opt_diam_1} is not a \emph{cone program} in the sense of \cite[\S4.6.1]{BoydVandenberghe:2004};  that term refers instead to the minimization of a linear objective function over a closed convex cone that contains no lines and has non-empty interior.

A point $(x, y) \in \mathcal{X} \times \R$ such that $| y - G(z) | \leq d_{L}(x, z)$ is said to be \emph{feasible} with respect to the data point $(z, G(z))$;  if this holds for all $z \in \mathcal{O}$, then $(x, y)$ is said to be \emph{$G|_{\mathcal{O}}$-feasible}.

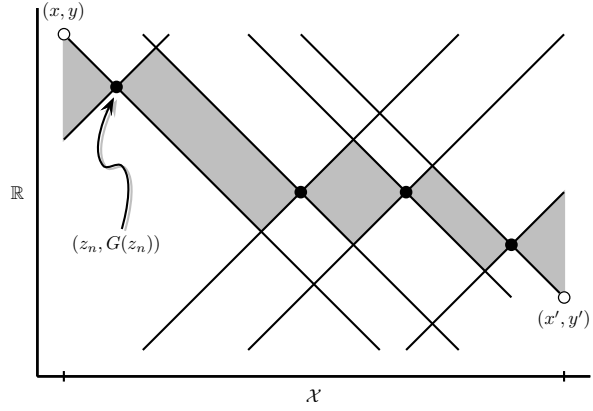
\begin{figure}
	\scalebox{0.7}{
		\begin{pspicture}(-6.5,-4)(6.0,4)
	\psset{linewidth=0.04cm}
	\psset{arrowsize=0 6}
	\psset{dotsize=0 6}
	\psset{shadowcolor=lightgray}
	\psset{shadowsize=2pt}
	\psset{shadowangle=-30}

	\pspolygon[dimen=middle, linecolor=lightgray, fillstyle=solid, fillcolor=lightgray](-5,3)(-5,1)(-4,2)
	\pspolygon[dimen=middle, linecolor=lightgray, fillstyle=solid, fillcolor=lightgray](-4,2)(-1.25,-0.75)(-0.5,0)(-3.25,2.75)
	\pspolygon[dimen=middle, linecolor=lightgray, fillstyle=solid, fillcolor=lightgray](-0.5,0)(0.5,-1)(1.5,0)(0.5,1)
	\pspolygon[dimen=middle, linecolor=lightgray, fillstyle=solid, fillcolor=lightgray](1.5,0)(3,-1.5)(3.5,-1)(2,0.5)
	\pspolygon[dimen=middle, linecolor=lightgray, fillstyle=solid, fillcolor=lightgray](3.5,-1)(4.5,-2.0)(4.5,0.0)
	
	\psdot(-4,2)
	\psline(-5,3)(1,-3)
	\psline(-5,1)(-3,3)
	\rput(-4,-1){$(z_{n}, G(z_{n}))$}
	\pscurve[shadow=true]{->}(-3.9,-0.7)(-3.9,0.5)(-4.25,0.5)(-4.05,1.8)
	
	\psdot(-0.5,0)
	\psline(-3.5,3)(2.5,-3)
	\psline(-3.5,-3)(2.5,3)

	\psdot(1.5,0)
	\psline(-1.5,3)(3.5,-2.0)
	\psline(-1.5,-3)(4.5,3)

	\psdot(3.5,-1)
	\psline(-0.5,3)(4.5,-2.0)
	\psline(1.5,-3)(4.5,0.0)
	
	\psdots[dotstyle=Bo, fillcolor=white](-5,3)(4.5,-2.0)
	\rput[b](-5,3.25){$(x, y)$}
	\rput[t](4.5,-2.25){$(x', y')$}
	
	\psaxes[labels=none, ticks=none](-5.5,-3.5)(-5.5,-3.5)(5.0,3.5)
	\rput[r](-5.75,0){$\R$}
	\psline(-5,-3.6)(-5,-3.4)
	\psline(4.5,-3.6)(4.5,-3.4)
	\rput[t](-0.25,-3.75){$\mathcal{X}$}
\end{pspicture}
	}
	\caption{Shaded, the feasible set for $(x, y)$ and $(x', y')$ in $\mathcal{X} \times \R$ for the optimization problem \eqref{eq:opt_diam_1} given the four observations represented by the four black dots.  The white dots show the optimal values for $(x, y)$ and $(x', y')$.  Note that, as well as being constrained to lie in the shaded feasible set, $(x, y)$ and $(x', y')$ must also satisfy $| y - y' | \leq d_{L}(x, x') \equiv L_{k} d_{k}(x^{k}, x'^{k})$.}
	\label{fig:feasible}
\end{figure}

Let $\widehat{D}_{k}[\mathcal{X}, G|_{\mathcal{O}}, d_{L}]$ (or simply $\widehat{D}_{k}$) denote the upper bound on $\mathcal{D}_{k}[G]$ that arises as the solution (extreme value) of the optimization problem \eqref{eq:opt_diam_1}.  It is natural to ask whether or not $\widehat{D}_{k}$ is the \emph{least} upper bound on $\mathcal{D}_{k}[G]$ given $G|_{\mathcal{O}}$ and $L$.  In fact, this is the case, and the proof relies on McShane's extension theorem.

\begin{thm}[Optimality of $\widehat{D}_{k}$]
	\label{thm:Dhat_optimal}
	Let
	\begin{equation}
		\label{eq:extend}
		\Extend(\mathcal{X}, G|_{\mathcal{O}}, d_{L}) := \{ g \colon \mathcal{X} \to \R \mid \text{$g$ is $d_{L}$-short and $g = G$ on $\mathcal{O}$} \}
	\end{equation}
	denote the set of all functions on $\mathcal{X}$ that have Lipschitz constant $L$ and interpolate the given values of $G$ on $\mathcal{O}$.  Then the maximum value $\widehat{D}_{k}$ of \eqref{eq:opt_diam_1} is the optimal upper bound on $\mathcal{D}$ given $G|_{\mathcal{O}}$ and $L$ in the sense that
	\begin{equation}
		\label{eq:Dhat_optimal}
		\widehat{D}_{k} = \sup \{ \mathcal{D}_{k}[g] \mid g \in \Extend(\mathcal{X}, G|_{\mathcal{O}}, d_{L}) \}.
	\end{equation}
\end{thm}

\begin{proof}
	Let $S$ denote the supremum on the right-hand side of \eqref{eq:Dhat_optimal}.  Suppose that $\widehat{D}_{k} > S$.  Then there exist points $(x, y)$ and $(x', y') \in \mathcal{X} \times \R$ that satisfy the constraints in \eqref{eq:opt_diam_1} and the inequality
	\[
		S < | y - y' | \leq \widehat{D}_{k}.
	\]
	Define $g \colon \mathcal{O} \cup \{ x, x' \} \to \R$ by
	\begin{align*}
		g(z) &:= G(z) \text{ for each } z \in \mathcal{O}, \\
		g(x) &:= y, \\
		g(x') &:= y'.
	\end{align*}
	This $g$ is $d_{L}$-short, and so McShane's extension theorem implies that $g$ can be extended to some $d_{L}$-short $\bar{g} \colon \mathcal{X} \to \R$.  Necessarily, $\bar{g}|_{\mathcal{O}} = g|_{\mathcal{O}} = G|_{\mathcal{O}}$.  However, by construction, $\mathcal{D}_{k}[\bar{g}] \geq | y - y' | > S$, which contradicts the definition of $S$.  Hence, by contradiction, $\widehat{D}_{k} \leq S$.
	
	Now suppose that $\widehat{D}_{k} < S$.  Then there exists some $d_{L}$-short $g \colon \mathcal{X} \to \R$ such that $g = G$ on $\mathcal{O}$ and $\mathcal{D}_{k}[g] > \widehat{D}_{k}$;  hence, there exist points $x, x' \in \mathcal{X}$ that differ only in their $k^{\text{th}}$ component and such that
	\[
		\widehat{D}_{k} < | g(x) - g(x') | \leq \mathcal{D}_{k}[g].
	\]
	However, $x$ and $x'$ with $y := g(x)$ and $y' := g(x')$ satisfy the constraints in \eqref{eq:opt_diam_1}, and so $\widehat{D}_{k} \geq | g(x) - g(x') |$, which is a contradiction.  Hence, $\widehat{D}_{k} \geq S$, which completes the proof.
\end{proof}

\begin{rmk}
	It is important to note that although $\widehat{D}_{k}$ is the optimal upper bound on $\mathcal{D}_{k}[G]$ given $G|_{\mathcal{O}}$ and $L$, and hence $\widehat{D} := \big( \widehat{D}_{1}^{2} + \dots + \widehat{D}_{K}^{2} \big)^{1/2} \geq \mathcal{D}[G]$, it is not generally true that $\widehat{D}$ is the optimal upper bound on $\mathcal{D}[G]$ given the same information ($G|_{\mathcal{O}}$ and $L$).  The reason for this is that the (approximate) maximizers for, say, $\widehat{D}_{1}$ and $\widehat{D}_{K}$ may not be mutually consistent, \emph{i.e.}\ $d_{L}$-short.  
	
	Note also that the upper bound
	\[
		\P[G(X) \leq \theta] \leq \exp \left( - \frac{2 (\E[G(X)] - \theta)_{+}^{2}}{\widehat{D}_{1}^{2} + \dots + \widehat{D}_{K}^{2}} \right)
	\]
	is not the least upper bound on $\P[G(X) \leq \theta]$ given $\E[G(X)$] and that $\mathcal{D}_{k}[G] \leq \widehat{D}_{k}$.  The optimal such bound is given by the optimal McDiarmid inequality of \cite[\S4]{OwhadiScovelSullivanMcKernsOrtiz:2010}.
	
	\COMMENT{
		There are two solutions to this deficiency:
			\begin{enumerate}
				\item \emph{an integrated problem for $\mathcal{D}[G]$:}  one solution is to solve directly for the least upper bound on $\mathcal{D}[G]$ given $G|_{\mathcal{O}}$ and $L$, namely to maximize
				\[
					\left( \sum_{k = 1}^{K} | y_{k} - y'_{k} |^{2} \right)^{1/2}
				\]
				with respect to $\big( (x_{k}, y_{k}), (x'_{k}, y'_{k}) \big)_{k = 1}^{K}$ such that, for each $k \in \{ 1, \dots, K \}$ and $i \in \{ 1, \dots, K \} \setminus \{ k \}$, $x_{k}^{i} = {x'}_{k}^{i}$, and such that all the candidate values are mutually $d_{L}$-short and $d_{L}$-short with respect to $\mathcal{O}$.  Clearly, this problem is harder to solve than \eqref{eq:opt_diam_1} since it is posed in a space of dimension $K$ times greater than the problem \eqref{eq:opt_diam_1}.
				\item \emph{parallel/distributed computing:}  alternatively, $K$ instances of the problem \eqref{eq:opt_diam_1} can be run in parallel, with additional dynamically updated constraints to ensure that each $(x_{k}, y_{k})$ and $(x'_{k}, y'_{k})$ is $d_{L}$-short not only with respect to each $(z, G(z))$ but also with respect to each $(x_{\ell}, y_{\ell})$ and $(x'_{\ell}, y'_{\ell})$ for $\ell \neq k$.
			\end{enumerate}
			Mathematically, these two formulations are identical;  they differ only in their computational practicalities.
	}
\end{rmk}

\subsection{Error Bounds}

In addition to the question of optimality, it is natural to ask whether or not solutions $\widehat{D}_{k}$ of \eqref{eq:opt_diam_1} converge to the McDiarmid subdiameter $\mathcal{D}_{k}[G]$ as the number of observations increases to infinity.  Unsurprisingly, the important quantity is not the number of observations, but rather the largest gap between them, as measured by the metric $d_{L}$.  Define the \emph{gap size} of the observation set $\mathcal{O}$ on $\mathcal{X}$ to be the (asymmetric) Hausdorff distance from $\mathcal{X}$ to $\mathcal{O}$ with respect to $d_{L}$, \emph{i.e.}
\begin{equation}
	\label{eq:gap_size}
	\Gamma(\mathcal{X}, \mathcal{O}, d_{L}) 
	:= \sup_{x \in \mathcal{X}} d_{L}(x, \mathcal{O})
	:= \sup_{x \in \mathcal{X}} \inf_{z \in \mathcal{O}} d_{L}(x, z).
\end{equation}

\begin{thm}[Error bound for $\widehat{D}_{k}$]
	\label{thm:Dhat_error}
	For any $G \colon \mathcal{X} \to \R$ with finite McDiarmid subdiameter $\mathcal{D}_{k}[G]$ and any $\mathcal{O} \subseteq \mathcal{X}$,
	\begin{equation}
		\label{eq:Dhat_error}
		0 \leq \widehat{D}_{k} - \mathcal{D}_{k}[G] \leq 4 \Gamma(\mathcal{X}, \mathcal{O}, d_{L}).
	\end{equation}
\end{thm}

\begin{proof}
	Theorem \ref{thm:Dhat_optimal} shows that $\widehat{D}_{k} \geq \mathcal{D}_{k}[G]$, so it remains to show the effective ``$4 \Gamma$'' part of the error estimate.
	
	Let $\eps > 0$ be arbitrary.  Let $(x, y)$ and $(x', y') \in \mathcal{X} \times \R$ satisfy the constraints in \eqref{eq:opt_diam_1} and be $\eps$-approximate maximizers for that problem, \emph{i.e.}
	\[
		\widehat{D}_{k} - \eps \leq | y - y' | \leq \widehat{D}_{k}.
	\]
	Then, even though the values $G(x)$ and $G(x')$ may be unknown since $x$ and $x'$ are not necessarily members of $\mathcal{O}$, the following estimate holds:
	\begin{align*}
		\widehat{D}_{k} 
		& \leq | y - y' | + \eps \\
		& \leq | y - G(x) | + | G(x) - G(x') | + | G(x') - y' | + \eps \\
		& \leq | y - G(x) | + \mathcal{D}_{k}[G] + | G(x') - y' | + \eps.
	\end{align*}
	
	By the definition of the gap size, there exists some $z \in \mathcal{O}$ such that $d_{L}(x, z) \leq \Gamma(\mathcal{X}, \mathcal{O}, d_{L})$, and so
	\begin{align*}
		| y - G(x) |
		& \leq | y - G(z) | + | G(z) - G(x) | \\
		& \leq d_{L}(x, z) + d_{L}(z, x) \\
		& \leq 2 \Gamma(\mathcal{X}, \mathcal{O}, d_{L}).
	\end{align*}
	Similarly, there exists $z' \in \mathcal{O}$ such that $d_{L}(x', z') \leq \Gamma(\mathcal{X}, \mathcal{O}, d_{L})$, and so $| G(x') - y' | \leq 2 g$.  Therefore, $\widehat{D}_{k} \leq \eps + 4 \Gamma(\mathcal{X}, \mathcal{O}, d_{L}) + \mathcal{D}_{k}[G]$ and, since $\eps > 0$ was arbitrary, the claim follows.
\end{proof}

\subsection{Structure of the Feasible Set}

The constrained optimization problem \eqref{eq:opt_diam_1} entails exploration of the feasible set $\Extend(\mathcal{X}, G|_{\mathcal{O}}, d_{L})$ of $d_{L}$-short extensions of the data $G|_{\mathcal{O}}$ to all of $\mathcal{X}$:
\[
	\Extend(\mathcal{X}, G|_{\mathcal{O}}, d_{L}) := \{ g \colon \mathcal{X} \to \R \mid \text{$g$ is $d_{L}$-short and $g = G$ on $\mathcal{O}$} \}.
\]
Note that $\Extend(\mathcal{X}, G|_{\mathcal{O}}, d_{L})$ not a linear space, but is a convex subset of the linear space of all real-valued functions on $\mathcal{X}$.  Furthermore, by the Arzel{\`a}--Ascoli theorem, $\Extend(\mathcal{X}, G|_{\mathcal{O}}, d_{L})$ is complete with respect to the uniform (supremum) norm, and is compact whenever $\mathcal{X}$ is compact.  McShane's extension theorem (Theorem \ref{thm:McShane}) is the assertion that, whenever $G|_{\mathcal{O}}$ has Lipschitz constant $L$ on $\mathcal{O}$, $\Extend(\mathcal{X}, G|_{\mathcal{O}}, d_{L})$ is non-empty.  Theorem \ref{thm:Dhat_optimal} states that the maximum value of $g \mapsto \mathcal{D}_{k}[g]$ over $g \in \Extend(\mathcal{X}, G|_{\mathcal{O}}, d_{L})$ can be found by restricting attention to a finite-dimensional subset as described by the constraints in \eqref{eq:opt_diam_1}.  Indeed, this search can be made even simpler than \eqref{eq:opt_diam_1} suggests by considering structure of the problem for $y$ and $y'$ with fixed $x$ and $x'$.

For fixed $x \in \mathcal{X}$, define the least and greatest feasible values of $g(x)$ among $g \in \Extend(\mathcal{X}, G|_{\mathcal{O}}, d_{L})$ by
\begin{align*}
	Y^{-}(x, G|_{\mathcal{O}}, L) & := \sup_{g \in \Extend(\mathcal{X}, G|_{\mathcal{O}}, d_{L})} g(x) = \sup_{z \in \mathcal{O}} G(z) - d_{L}(x, z), \\
	Y^{+}(x, G|_{\mathcal{O}}, L) & := \inf_{g \in \Extend(\mathcal{X}, G|_{\mathcal{O}}, d_{L})} g(x) = \inf_{z \in \mathcal{O}} G(z) + d_{L}(x, z),
\end{align*}
Note that these quantities are easily calculated when $\mathcal{O}$ is a finite set.  The pair $(y, y') \in \R^{2}$ is feasible (\emph{i.e.}\ $y = g(x)$ and $y' = g(x')$ for some $x, x' \in \mathcal{X}$ that differ only in their $k^{\text{th}}$ component) if, and only if,
\begin{align*}
	y & \in \big[ Y^{-}(x, G|_{\mathcal{O}}, L), Y^{+}(x, G|_{\mathcal{O}}, L) \big] \text{,} \\
	y' & \in \big[ Y^{-}(x', G|_{\mathcal{O}}, L), Y^{+}(x', G|_{\mathcal{O}}, L) \big] \text{, and} \\
	| y - y' | & \leq d_{L}(x, x') = L_{k} d_{k}(x^{k}, x'^{k}) .
\end{align*}
So, for each $(x, x')$, the set of feasible $(y, y')$ is a closed and convex polygon in $\R^{2}$.  The maximum value of $| y - y' |$ over this polygon is $A(x, x')$, defined by
\[
	A(x, x') := \min \begin{Bmatrix} d_{L}(x, x'), \\ Y^{+}(x, G|_{\mathcal{O}}, L) - Y^{-}(x', G|_{\mathcal{O}}, L), \\ Y^{+}(x', G|_{\mathcal{O}}, L) - Y^{-}(x, G|_{\mathcal{O}}, L) \end{Bmatrix}.
\]
The constrained optimization problem \eqref{eq:opt_diam_1} is, therefore, equivalent to the following unconstrained (and, therefore, more easily solved) problem in $K + 1$ variables $x^{1}, \dots, x^{k}, x'^{k}, \dots x^{K}$:
\begin{equation}
	\label{eq:opt_diam_easier}
	\begin{cases}
		\text{maximize: } & A(x, x') \text{;} \\
		\text{among: } 
		& x \in \mathcal{X} \\
		& x'^{k} \in \mathcal{X}_{k} \\
		& x' := (x^{1}, \dots, x^{k - 1}, x'^{k}, x^{k + 1}, \dots, x^{K}) \text{.} \\
	\end{cases}
\end{equation}

\subsection{Examples}

As a simple example that can be solved explicitly, consider an affine function $G \colon \mathcal{X} := [0, 1]^{K} \to \R$:
\begin{equation}
	\label{eq:G_affine}
	G(x) = a_{0} + \sum_{k = 1}^{K} a_{k} x^{k}
\end{equation}
for some constants $a_{0}, a_{1}, \dots, a_{K} \in \R$.  Suppose that the observation set $\mathcal{O}$ consists of a $N_{1} \times \dots \times N_{K}$ rectangular grid of equally-spaced points of $[0, 1]^{K}$, with observations at the corners of the cube.  Given $L_{k} \geq | a_{k} |$, the gap size for this observation set is
\begin{equation}
	\label{eq:gap_affine}
	\Gamma(\mathcal{X}, \mathcal{O}, d_{L}) = \sum_{k = 1}^{K} \frac{L_{k}}{2 (N_{k} - 1)}.
\end{equation}
The exact McDiarmid subdiameters of $G$ satisfy $\mathcal{D}_{k}[G] = | a_{k} |$.  On the other hand, $\widehat{D}_{k}$, the least upper bound on $\mathcal{D}_{k}[G]$ given the observations $G|_{\mathcal{O}}$ and the Lipschitz constants $L_{1}, \dots, L_{K}$ but \emph{not} the information that $G$ is affine,\footnote{If $G$ is known to be affine and its values are given at $K + 1$ points in general position in $[0, 1]^{K}$, then $G$ is determined everywhere.} is given by
\begin{equation}
	\label{eq:Dhat_affine}
	\widehat{D}_{k} = | a_{k} | + \sum_{i = 1}^{K} \frac{L_{i} - | a_{i} |}{N_{i} - 1}.
\end{equation}
In this case, the error $\widehat{D}_{k} - \mathcal{D}_{k}[G]$ is approximately half the upper bound given by Theorem \ref{thm:Dhat_error} if $L_{k} \gg | a_{k} |$, and vanishes if $L_{k} = | a_{k} |$.

\section{Optimal Bounds on Probabilities}
\label{sec:opt_prob}

In this section, in the spirit of \cite{BertsimasPopescu:2005, OwhadiScovelSullivanMcKernsOrtiz:2010}, the emphasis is on providing optimal bounds on the probability of failure $\P[G(X) \leq \theta]$ rather than bounds on the McDiarmid diameter $\mathcal{D}[G]$.  Theorem \ref{thm:Dhat_optimal} shows that the optimization problem \eqref{eq:opt_diam_1} determines the optimal upper bound on each McDiarmid subdiameter $\mathcal{D}_{k}[G]$, and hence --- given that $\E[G(X)] \geq m$ and via McDiarmid's inequality \eqref{eq:McDiarmid_fleqtheta} --- an upper bound on the probability of failure $\P[G(X) \leq \theta]$.  However, this bound is not necessarily the sharpest one given the available information, namely that $G$ is $d_{L}$-short, its inputs are independent, and that $G|_{\mathcal{O}}$ and $\E[G(X)]$ are as given.  The optimal upper bound on the probability of failure given this information is denoted by $\widehat{P}[\mathcal{X}, G|_{\mathcal{O}}, L, m]$ (or simply $\widehat{P}$) and is given by
\begin{equation}
	\label{eq:opt_prob}
	\widehat{P} := \sup_{(g, \mu) \in \mathcal{A}} \mu[g \leq \theta],
\end{equation}
where
\begin{equation}
	\label{eq:opt_prob_feasible}
	\mathcal{A} := \left\{ (g, \mu)
	\smid
	\begin{array}{c}
		g \colon \mathcal{X} \to \R \text{ is $d_{L}$-short}, \\
		\mu = \mu_{1} \otimes \dots \otimes \mu_{K} \in \bigotimes_{k = 1}^{K} \mathcal{P}(\mathcal{X}_{k}), \\
		g = G \text{ on $\mathcal{O}$, and } \E_{\mu}[g] \geq m
	\end{array}
	\right\} \text{,}
\end{equation}
\emph{i.e.}
\[
	\mathcal{A} := \left\{ (g, \mu)
	\smid
	\begin{array}{c}
		\mu = \mu_{1} \otimes \dots \otimes \mu_{K} \in \bigotimes_{k = 1}^{K} \mathcal{P}(\mathcal{X}_{k}), \\
		\text{$g \in \Extend(\mathcal{X}, G|_{\mathcal{O}}, d_{L})$, and } \E_{\mu}[g] \geq m
	\end{array}
	\right\} \text{.}
\]

This infinite-dimensional optimization problem over coupled $g \in \Extend(\mathcal{X}, G|_{\mathcal{O}}, d_{L})$ and $\mu \in \mathcal{P}(\mathcal{X})$ is more numerically tractable that it may seem.  The next subsection shows that, for each $g$, the extreme values can be found by searching only among measures $\mu$ that have a particularly simple structure;  furthermore, this simple structure simplifies the search over $g$ as well.  

\begin{rmk}
	Note that while the examples below have only two constraints on the measure $\mu$, namely the product structure and that $\E_{\mu}[g] \geq m$, \emph{any} combination of information on independence, non-independence, correlations and generalized moments can be used in the same way.  For further discussion, see the general theory expounded in \cite{OwhadiScovelSullivanMcKernsOrtiz:2010} and the remarks in Subsection \ref{subsec:additional_info}.
\end{rmk}

\subsection{Finite-Dimensional Reduction Theorem}

Given two points $x_{0}, x_{1} \in \mathcal{X}$, let $\mathcal{C}(x_{0}, x_{1})$ denote the discrete cube in $\mathcal{X}$ that has $x_{0}$ and $x_{1}$ as its ``opposite corners'':
\begin{equation}
	\label{eq:cube}
	\mathcal{C}(x_{0}, x_{1}) := \left\{ x_{\eps} := \left( x_{\eps_{k}}^{k} \right)_{k = 1}^{K} \in \mathcal{X} \smid \eps \in \{ 0, 1 \}^{K} \right\}.
\end{equation}
The elements of $\mathcal{C}(x_{0}, x_{1})$ are indexed by the elements of the Hamming cube $\{ 0, 1 \}^{K}$:  for $\eps \in \{ 0, 1 \}^{K}$, $x_{\eps} \in \mathcal{C}(x_{0}, x_{1})$ is the point whose $k^{\text{th}}$ component is the same as the $k^{\text{th}}$ component of $x_{0}$ if $\eps_{k} = 0$, and the same as the $k^{\text{th}}$ component of $x_{1}$ if $\eps_{k} = 1$.

Recall that a topological space $\mathcal{Z}$ is said to be a \emph{Radon space} if it is separable and every Borel probability measure on $\mathcal{Z}$ is inner regular \cite{Schwartz:1973, Wage:1980};  that is, $\mathcal{Z}$ is a Radon space if it has a countable dense subset and, for every $\mu \in \mathcal{P}(\mathcal{Z})$ and every Borel-measurable set $B \subseteq \mathcal{Z}$,
\begin{equation}
	\label{eq:inner_regular}
	\mu(B) = \sup \{ \mu(K) \mid K \subseteq B \text{ and $K$ is compact} \}.
\end{equation}
In particular, any continuous Hausdorff image of a separable and completely metrizable space (a \emph{Suslin space}) is a Radon space.  Compact subsets of Euclidean space $\R^{n}$ are Radon spaces, whereas a simple example of a non-inner-regular probability measure (and hence a non-Radon space) is $[0, 1]$ with the topology of convergence from the right \cite[Ex.~51]{SteenSeebach:1978} and uniform (Lebegsue) measure.

Under the mild technical assumption that each $(\mathcal{X}_{k}, d_{k})$ is a Radon space, the reduction theorems of \cite{OwhadiScovelSullivanMcKernsOrtiz:2010} imply that, for each $d_{L}$-short $g \colon \mathcal{X} \to \R$, the extreme value in \eqref{eq:opt_prob} is obtained among product probability measures $\mu$ such that each marginal distribution $\mu_{k}$ has support on at most two points of $\mathcal{X}_{k}$ --- \emph{i.e.}\ $\mu_{k}$ is a convex combination of at most two Dirac measures (point masses).  That is, it is sufficient to search over probability measures of the form
\begin{equation}
	\label{eq:prod_representation}
	\mu = \bigotimes_{k = 1}^{K} \mu_{k} = \bigotimes_{k = 1}^{K} \left( p_{k} \delta_{x_{0}^{k}} + (1 - p_{k}) \delta_{x_{1}^{k}} \right)
\end{equation}
that are supported on $\mathcal{C}(x_{0}, x_{1})$ for some $x_{0}, x_{1} \in \mathcal{X}$;  $x_{0}$, $x_{1}$ and $p$ are parameters with respect to which we must optimize.

It is a simple matter of combinatorics to convert the product representation \eqref{eq:prod_representation} into the sum representation
\begin{equation}
	\label{eq:sum_representation}
	\mu = \sum_{\eps \in \{ 0, 1 \}^{K}} \left( \prod_{k = 1}^{K} (p_{k})^{1 - \eps_{k}} (1 - p_{k})^{\eps_{k}} \right) \delta_{x_{\eps}}
\end{equation}
using the indexing scheme \eqref{eq:cube}.  If $\mu$ is any such measure and $r$ is any real-valued measurable function defined on any superset of $\mathcal{C}(x_{0}, x_{1})$, then $\E_{\mu}[r]$ exists and depends only upon the points $x_{\eps}$, the values $y_{\eps} := g(x_{\eps})$ and the weights $p_{k}$.  The sum representation \eqref{eq:sum_representation} makes the calculation of $\E_{\mu}[r]$ very easy:
\begin{equation}
	\label{eq:sum_representation_r}
	\E_{\mu}[r] = \sum_{\eps \in \{ 0, 1 \}^{K}} \left( \prod_{k = 1}^{K} (p_{k})^{1 - \eps_{k}} (1 - p_{k})^{\eps_{k}} \right) r(x_{\eps}).
\end{equation}
In particular, given $g \colon \mathcal{X} \to \R$, the mean and probability of failure for $g$ are easily calculated using \eqref{eq:sum_representation_r} with $r = g$ and $r = \one[g \leq \theta]$ respectively.

As the following theorem shows, a search over the finite-dimensional collection of feasible $x_{0}$, $x_{1}$, $\{ y_{\eps} \mid \eps \in \{ 0, 1 \}^{K} \}$ and $p \in [0, 1]^{K}$ has the same extreme values as the infinite-dimensional problem \eqref{eq:opt_prob}--\eqref{eq:opt_prob_feasible}, where ``feasible'' means being $d_{L}$-short, extending $G|_{\mathcal{O}}$, and having the right mean value:

\begin{thm}[Optimality/finite-dimensional reduction]
	\label{thm:Phat_optimal}
	Suppose that $(\mathcal{X}_{k}, d_{k})$ is a Radon space for each $k \in \{ 1, \dots, K \}$.  Let $\mathcal{A}$ be given by \eqref{eq:opt_prob_feasible} and let
	\begin{equation}
		\label{eq:opt_prob_reduced_feasible}
		\mathcal{A}_{\Delta} := \left\{ (g, \mu) \smid \begin{matrix}
			\text{for some $x_{0}, x_{1} \in \mathcal{X}$,} \\
			g \colon \mathcal{C}(x_{0}, x_{1}) \cup \mathcal{O} \to \R \text{ is $d_{L}$-short,} \\
			\mu = \bigotimes_{k = 1}^{K} \mu_{k} \in \mathcal{P}(\mathcal{C}(x_{0}, x_{1})) \cap \bigotimes_{k = 1}^{K} \mathcal{P}(\mathcal{X}_{k}), \\
			g = G \text{ on $\mathcal{O}$, and } \E_{\mu}[g] \geq m
		\end{matrix} \right\}.
	\end{equation}
	Then
	\begin{equation}
		\label{eq:Phat_opt_dim}
		\dim(\mathcal{A}_{\Delta}) = 2 \sum_{k = 1}^{K} \dim(\mathcal{X}_{k}) + 2^{K} + K \text{,}
	\end{equation}
	\begin{equation}
		\label{eq:Phat_opt_sup}
		\sup_{(g, \mu) \in \mathcal{A}} \mu[g \leq \theta] = \sup_{(g, \mu) \in \mathcal{A}_{\Delta}} \mu[g \leq \theta] \text{,}
	\end{equation}
	\begin{equation}
		\label{eq:Phat_opt_inf}
		\inf_{(g, \mu) \in \mathcal{A}} \mu[g \leq \theta] = \inf_{(g, \mu) \in \mathcal{A}_{\Delta}} \mu[g \leq \theta] \text{.}
	\end{equation}
\end{thm}

\begin{proof}
	Assertion \eqref{eq:Phat_opt_dim} follows from the fact that an element of $\mathcal{A}_{\Delta}$ is determined by a choice of $x_{0} \in \mathcal{X}$, $x_{1} \in \mathcal{X}$, $p \in [0, 1]^{K}$ as in \eqref{eq:prod_representation} or \eqref{eq:sum_representation}, and a choice of $g(x)$ for each of the $2^{K}$ points of $\mathcal{C}(x_{0}, x_{1})$.
	
	To prove \eqref{eq:Phat_opt_sup}, let $S := \sup_{(g, \mu) \in \mathcal{A}} \mu[g \leq \theta]$.  Then
	\begin{align*}
		S
		&= \sup \left\{ \mu[g \leq \theta] \smid \begin{matrix} \text{for some $x_{0}, x_{1} \in \mathcal{X}$,} \\ g \colon \mathcal{X} \to \R \text{ is $d_{L}$-short,} \\ \mu = \bigotimes_{k = 1}^{K} \mu_{k} \in \mathcal{P}(\mathcal{C}(x_{0}, x_{1})) \cap \bigotimes_{k = 1}^{K} \mathcal{P}(\mathcal{X}_{k}), \\ g = G \text{ on $\mathcal{O}$, and } \E_{\mu}[g] \geq m \end{matrix} \right\} \\
		& \leq \sup \left\{ \mu[g \leq \theta] \smid \begin{matrix} \text{for some $x_{0}, x_{1} \in \mathcal{X}$,} \\ g \colon \mathcal{C}(x_{0}, x_{1}) \cup \mathcal{O} \to \R \text{ is $d_{L}$-short,} \\ \mu = \bigotimes_{k = 1}^{K} \mu_{k} \in \mathcal{P}(\mathcal{C}(x_{0}, x_{1})) \cap \bigotimes_{k = 1}^{K} \mathcal{P}(\mathcal{X}_{k}), \\ g = G \text{ on $\mathcal{O}$ and } \E_{\mu}[g] \geq m \end{matrix} \right\} \\
		& = \sup_{(g, \mu) \in \mathcal{A}_{\Delta}} \mu[g \leq \theta].
	\end{align*}
	The first equality follows from the reduction theorem \cite[Theorem 3.1 and Corollary 3.4]{OwhadiScovelSullivanMcKernsOrtiz:2010} and the inequality follows from the fact that only the values of $g$ on the discrete cube $\mathcal{C}(x_{0}, x_{1})$ are germane to the probability of failure and the mean constraint;  the final equality holds true by definition of the right-hand side.
	
	To see that this inequality must, in fact, be an equality, suppose for a contradiction that $S < \sup_{(g, \mu) \in \mathcal{A}_{\Delta}} \mu[g \leq \theta]$.  Then there exist some $x_{0}, x_{1} \in \mathcal{X}$, $p \in [0, 1]^{K}$ and a $d_{L}$-short $g \colon \mathcal{C}(x_{0}, x_{1}) \cup \mathcal{O} \to \R$ such that $g = G$ on $\mathcal{O}$, $\E_{\mu}[g] \geq m$ and $\mu[g \leq \theta] > S$.  By McShane's extension theorem, there exists an extension of $g$ to a $d_{L}$-short function $\bar{g} \colon \mathcal{X} \to \R$;  necessarily, this extension has $\bar{g} = G$ on $\mathcal{O}$, $\E_{\mu}[\bar{g}] = \E_{\mu}[g] \geq m$ and $\mu[\bar{g} \leq \theta] = \mu[g \leq \theta] > S$, \emph{i.e.}\ $(\mu, \bar{g}) \in \mathcal{A}$.  Hence, $S < \mu[\bar{g} \leq \theta] \leq S$, which is a contradiction.
	
	This establishes \eqref{eq:Phat_opt_sup};  the proof of \eqref{eq:Phat_opt_inf} is similar, and is omitted.
\end{proof}

Theorem \ref{thm:Phat_optimal} shows that the infinite-dimensional optimization problem \eqref{eq:opt_prob} is equivalent to (\emph{i.e.}\ has the same extreme value as) the following finite-dimensional optimization problem, where now $y_{\eps}$ is written in place of $g(x_{\eps})$:
\begin{equation}
	\label{eq:opt_prob_reduced}
	\begin{cases}
		\text{maximize: } & \displaystyle \sum_{\eps \in \{ 0, 1 \}^{K}} \left( \prod_{k = 1}^{K} (p_{k})^{1 - \eps_{k}} (1 - p_{k})^{\eps_{k}} \right) \one[y_{\eps} \leq \theta] \text{;} \\
		\text{among: } & x_{0}, x_{1} \in \mathcal{X} \text{,} \\
			& y \colon \{ 0, 1 \}^{K} \to \R \text{,} \\
			& p \in [0, 1]^{K} \text{;} \\
		\text{subject to: } & \text{for all } \eps, \eps' \in \{ 0, 1 \}^{K}, \eps \neq \eps' \text{:} \\
			& \quad | y_{\eps} - y_{\eps'} | \leq d_{L}(x_{\eps}, x_{\eps'}); \\
			& \text{for all } \eps \in \{ 0, 1 \}^{K}, z \in \mathcal{O} \text{:} \\
			& \quad | y_{\eps} - G(z) | \leq d_{L}(x_{\eps}, z) \text{;} \\
			& \displaystyle \sum_{\eps \in \{ 0, 1 \}^{K}} \left( \prod_{k = 1}^{K} (p_{k})^{1 - \eps_{k}} (1 - p_{k})^{\eps_{k}} \right) y_{\eps} \geq m \text{.}
	\end{cases}
\end{equation}
See Figure \ref{fig:opt_prob_reduced} for a schematic illustration of the problem \eqref{eq:opt_prob_reduced}.  The problem \eqref{eq:opt_prob_reduced} has high dimension:  assuming that $\dim(\mathcal{X}_{k}) = 1$ for each $k$, \eqref{eq:opt_prob_reduced} is a problem in $3 K + 2^{K}$ unknowns with $2^{K - 1} (2^{K} - 1) + | \mathcal{O}| 2^{K} + 1$ distinct constraints.  However, as will be seen in Section \ref{sec:redundancy}, many of these constraints are redundant or non-binding.  Furthermore, we have numerical evidence that in some cases not all of the $2^{K}$ support points of the measure $\mu$ need to be considered:  see the remarks in Section \ref{sec:numerics} about ``dimensional collapse'' and Figure \ref{fig:num_results_collapse}.

\begin{figure}[t]
	\scalebox{0.7}{
		\begin{pspicture}(-6.0,-4.25)(12.0,3.5)
	\psset{linewidth=0.04cm}
	\psset{arrowsize=0 6}
	\psset{dotsize=0 6}
	\psset{shadowcolor=lightgray}
	\psset{shadowsize=2pt}
	\psset{shadowangle=-30}

	\pspolygon[dimen=middle, linewidth=0.01, linecolor=lightgray, fillstyle=solid, fillcolor=lightgray](-5.0,3.0)(-5.0,1.0)(-4.0,2.0)
	\pspolygon[dimen=middle, linewidth=0.01, linecolor=lightgray, fillstyle=solid, fillcolor=lightgray](-4,2)(-1.25,-0.75)(-0.5,0.0)(-3.25,2.75)
	\pspolygon[dimen=middle, linewidth=0.01, linecolor=lightgray, fillstyle=solid, fillcolor=lightgray](-0.5,0.0)(0.5,-1.0)(1.5,0.0)(0.5,1.0)
	\pspolygon[dimen=middle, linewidth=0.01, linecolor=lightgray, fillstyle=solid, fillcolor=lightgray](1.5,0.0)(3.0,-1.5)(3.5,-1.0)(2.0,0.5)
	\pspolygon[dimen=middle, linewidth=0.01, linecolor=lightgray, fillstyle=solid, fillcolor=lightgray](3.5,-1.0)(5.0,-2.5)(5.0,0.5)
	
	\psdot(-4.0,2.0)
	\psline(-5.0,3.0)(1.0,-3.0)
	\psline(-5.0,1.0)(-3.0,3.0)
	\rput[t](-4.0,-0.75){$\begin{array}{c} \text{data point/} \\ \text{observation} \\ (z, G(z)) \end{array}$}
	\pscurve[shadow=true]{->}(-3.925,-0.625)(-3.75,0.5)(-4.25,0.5)(-4.075,1.75)
	
	\psdot(-0.5,0.0)
	\psline(-3.5,3.0)(2.5,-3.0)
	\psline(-3.5,-3.0)(2.5,3.0)
	\psdot(1.5,0.0)
	\psline(-1.5,3.0)(4.5,-3.0)
	\psline(-1.5,-3.0)(4.0,2.5)
	\psdot(3.5,-1.0)
	\psline(-0.5,3.0)(5.0,-2.5)
	\psline(1.5,-3.0)(5.0,0.5)
	
	\psline[linestyle=dotted](8.0,-1.5)(10.0,-1.5)(10.0,1.0)(8.0,1.0)(8.0,-1.5)
	\psline[linewidth=0.25, linecolor=white](7.0,0.5)(9.0,0.5)(9.0,-2.0)
	\psline[linestyle=dotted](7.0,-2.0)(9.0,-2.0)(9.0,0.5)(7.0,0.5)(7.0,-2.0)
	\psline[linestyle=dotted](7.0,-2.0)(8.0,-1.5)
	\psline[linestyle=dotted](9.0,-2.0)(10.0,-1.5)
	\psline[linestyle=dotted](9.0,0.5)(10.0,1.0)
	\psline[linestyle=dotted](7.0,0.5)(8.0,1.0)
	
	\psdots[dotstyle=Bo, fillcolor=gray](7.0,-2.0)(9.0,-2.0)(7.0,0.5)(9.0,0.5)
	\psdots[dotstyle=Bo, fillcolor=gray](8.0,-1.5)(10.0,-1.5)(8.0,1.0)(10.0,1.0)

	\rput(6.5,-2.5){$x_{000}$}
	\rput(9.5,-2.5){$x_{100}$}
	\rput(6.5,1.0){$x_{001}$}
	\rput(7.5,1.5){$x_{011}$}
	\rput(10.5,-2.0){$x_{110}$}
	\rput(10.5,1.5){$x_{111}$}
	\rput(8.5,-0.5){$\mu$}
	
	\pscurve[linewidth=0.18, linecolor=white](6.75,-0.5)(4.75,-1.0)(4.95,-2.0)(3.2,-3.4)
	\pscurve[shadow=true, arrowsize=0 6]{->}(6.75,-0.5)(4.75,-1.0)(4.95,-2.0)(3.2,-3.4)
	
	\psaxes[labels=none, ticks=none](-5.75,-3.5)(-5.85,-3.6)(5.5,3.45)
	\psline[arrowsize=0 6]{->}(-5.75,3.0)(-5.75,3.55)
	\rput[r](-6.0,0.0){$\mathbb{R}$}
	\psline(-5.0,-3.6)(-5.0,-3.4)
	\psline(5.0,-3.6)(5.0,-3.4)
	\rput[t](0.0,-3.75){$\mathcal{X}$}
	
	\psdots[dotstyle=Bo, fillcolor=gray](-4.5,-3.5)(-3.0,-3.5)(-2.5,-3.5)(-0.875,-3.5)(1.0,-3.5)(3.0,-3.5)(4.0,-3.5)(4.5,-3.5)
	
	\psdots[dotstyle=Bo, fillcolor=white](-4.5,2.0)(-3.0,1.0)(-2.5,1.5)(-0.875,0.0)(1.0,0.25)(3.0,-1.0)(4.0,-1.0)(4.5,-0.75)
	
	\pscurve[linewidth=0.18, linecolor=white](6.0,2.625)(3.75,1.5)(3.25,2.0)(1.175,0.425)
	\pscurve[shadow=true]{->}(6.0,2.625)(3.75,1.5)(3.25,2.0)(1.175,0.425)
	\rput[l](3.0,3.0){\text{$(x_{010}, y_{010})$, with probability mass $p_{1} \cdot (1 - p_{2}) \cdot p_{3}$}}
	
\end{pspicture}
	}
	\caption{A schematic illustration of the variables in the optimization problem \eqref{eq:opt_prob_reduced}.  The black dots show the fixed locations of the legacy observations $G|_{\mathcal{O}}$.  The grey dots show the movable locations of the $2^{K}$ support points $x_{\eps}$, $\eps \in \{ 0, 1\}^{K}$, of the discrete product measure $\mu$ on $\mathcal{X}$.  The white dots show some feasible values $(x_{\eps}, y_{\eps})$.  The marginal distribution $\mu_{k}$ on $\mathcal{X}_{k}$ assigns mass $p_{k}$ to $x_{0}^{k}$ and mass $1 - p_{k}$ to $x_{1}^{k}$;  the mass of $x_{\eps}$ is determined by \eqref{eq:sum_representation}.}
	\label{fig:opt_prob_reduced}
\end{figure}
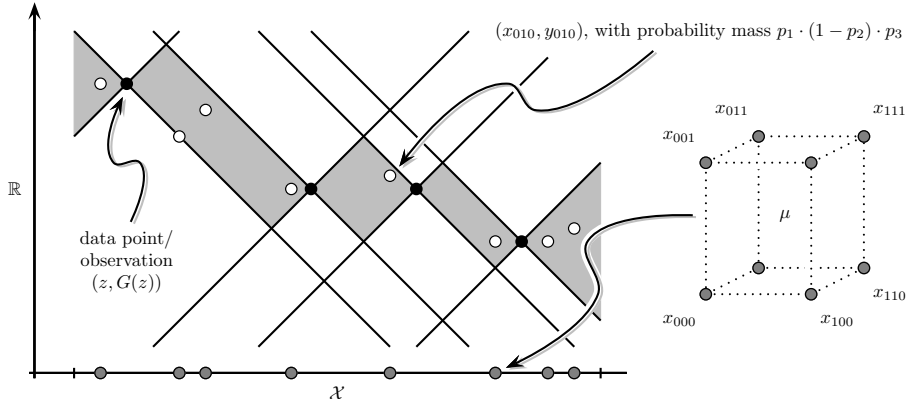

\subsection{Error Bounds}

As with the McDiarmid diameters, it is natural to ask how much of an over-estimate $\widehat{P}$ is of the true probability of failure $\P[G(X) \leq \theta]$.  Such an error estimate for the maximization problem \eqref{eq:opt_prob_reduced} is naturally provided by solving the corresponding \emph{minimization} problem.  That is, the double inequality
\[
	\inf_{(g, \mu) \in \mathcal{A}} \mu[g \leq \theta] \leq \P[G(X) \leq \theta] \leq \sup_{(g, \mu) \in \mathcal{A}} \mu[g \leq \theta]
\]
is \emph{ipso facto} the sharpest such inequality on the probability of failure given the available information encoded in $\mathcal{A}$ (\emph{i.e.}\ $G|_{\mathcal{O}}$, $L$ and $\E[G(X)] \geq m$).  It is not possible, on the basis of this information, to rule out the possibility that
\[
	\inf_{(g, \mu) \in \mathcal{A}} \mu[g \leq \theta] = \P[G(X) \leq \theta].
\]
Hence, the upper bound on $\widehat{P} - - \P[G(X) \leq \theta]$ is simply
\begin{equation}
	\label{eq:Phat_error}
	\widehat{P} - \P[G(X) \leq \theta] \leq \sup_{(g, \mu) \in \mathcal{A}} \mu[g \leq \theta] - \inf_{(g, \mu) \in \mathcal{A}} \mu[g \leq \theta],
\end{equation}
and this inequality is sharp, given the information encoded in $\mathcal{A}$.

\subsection{Prototypical Example}

The next example, Example \ref{eg:Phat_1dim_example}, in which \eqref{eq:opt_prob_reduced} is solved explicitly for one observation of a function on the unit interval, illustrates two very important points:  the least upper bound on the probability of failure, $\widehat{P}[\mathcal{X}, G|_{\mathcal{O}}, L, m]$, can depend discontinuously and non-monotonically on the observed data $G|_{\mathcal{O}}$.  It may be useful to first observe that
\begin{align*}
	\label{eq:seesaw}
	& \sup \left\{ \mu((-\infty, 0]) \smid 
	\begin{array}{c}
		\text{$\mu \in \mathcal{P}(\R)$, }
		\text{$\E_{Y \sim \mu}[Y] \geq m$,} \\
		\text{$\mu$ supported on an interval of length $\leq R$}
	\end{array}
	\right\} \\
	& \quad =
	\sup \left\{ \mu((-\infty, 0]) \smid 
	\begin{array}{c}
		\mu = p \delta_{y_{0}} + (1 - p) \delta_{y_{1}} \in \mathcal{P}(\R) \text{,} \\
		y_{0}, y_{1} \in \R \text{, } p \in [0, 1] \text{,} \\
		p y_{0} + (1 - p) y_{1} \geq m \text{,} \\
		| y_{0} - y_{1} | \leq R
	\end{array}
	\right\} \\
	& \quad = \left( 1 - \frac{m_{+}}{R} \right)_{+} \text{,}
\end{align*}
and that the maximizer satisfies $y_{0} = 0$, $y_{1} = R$.  The heuristic to bear in mind is that the event $[y_{0} = 0]$ can be assigned high probability if the value $y_{1}$ can be chosen to be sufficiently greater than the prescribed mean $m$.

\begin{eg}
	\label{eg:Phat_1dim_example}
	Suppose that a function $G \colon [0, 1] \to \R$ with Lipschitz constant $L > 0$ is observed at a single point, \emph{i.e.}\ $\mathcal{O} = \{ z \}$ for some $z \in [0, 1]$.  By symmetry, it is enough to consider the case that $z \in [0, \frac{1}{2}]$;  for simplicity, suppose that $G(z) > 0$;  also, it is no loss of generality to set the failure threshold to be $\theta := 0$.  
	
	Suppose it is known that $\E[G(X)] \geq m \in \R$;  necessarily, it must hold that $| G(z) - m | \leq L | 1 - z |$, otherwise the data and the mean and Lipschitz constraints are mutually contradictory.  The least upper bound $\widehat{P}$ on $\P[G(X) \leq 0]$ given the observation $(z, G(z))$, that $\E[G(X)] \geq m$, and the Lipschitz constant $L$, is given in five cases:
	\begin{equation}
		\label{eq:discontinuity}
		\widehat{P}
		=
		\begin{cases}
			\left( 1 - \tfrac{m_{+}}{L - (L z - G(z))} \right)_{+} \text{,} & \text{if $G(z) \leq L z$,} \\
			\left( 1 - \tfrac{m_{+}}{L - (L z + G(z))} \right)_{+} \text{,} & \text{if $L z < G(z) \leq L | \tfrac{1}{2} - z|$,} \\
			\left( 1 - \tfrac{2 m_{+}}{L + (G(z) - L z)} \right)_{+} \text{,} & \text{if $L | \tfrac{1}{2} - z| < G(z) \leq L | 1 - 3 z |$,} \\
			\left( 1 - \tfrac{m_{+}}{L z + G(z)} \right)_{+} \text{,} & \text{if $G(z) > L \max \{ z, 1 - 3 z \}$,} \\
			0, & \text{if $G(z) > L | 1 - z |$.}
		\end{cases} 
	\end{equation}
	The five cases are shown in Figure \ref{fig:discty};  surface and contour plots of $\widehat{P}$ as a function of the observed data $(z, G(z))$ were given in the introduction in Figure \ref{fig:discty-intro}.  Note well that $\widehat{P}$ is neither continuous nor monotone with respect to $(z, G(z))$:  the boundaries among the five cases define ``critical lines'' in data space, across which there are stark changes in the conclusions that may be inferred from the observed data.  Note also that the maximizers for \eqref{eq:opt_prob_reduced} may be non-unique:  \emph{e.g.}\ in Figure \ref{fig:discty}(a), which corresponds to the first case in \eqref{eq:discontinuity}, the maximum is attained by any $(x_{0}, y_{0})$, $(x_{1}, y_{1})$ and $p$ satisfying
	\begin{align*}
		x_{0} & \in [ 0,  z - G(z) / L], & y_{0} & = 0, \\
		x_{1} & = 1, & y_{1} & = L - L z + G(z), \\
		p & = \left( 1 - \tfrac{m_{+}}{| y_{1} - y_{0} |} \right)_{+}.
	\end{align*}
	There is a similar lack of uniqueness in Figure \ref{fig:discty}(d).  On the other hand, the maximizers in Figures \ref{fig:discty}(b) and (c) are unique.
		
	\begin{figure}
		\subfigure[$(z, G(z)) = (\frac{3}{8}, \frac{1}{4})$, and $\widehat{P} = \tfrac{3}{7}$]{
			\scalebox{0.7}{
				\begin{pspicture}(-0.5,-0.5)(4.5,4.0)
	\psset{linewidth=0.04}
	\psset{dotsize=0 6}
	
	\pspolygon[dimen=middle, linewidth=0.01, linecolor=lightgray, fillstyle=solid, fillcolor=lightgray](0.0,2.5)(1.5,1.0)(0.5,0.0)(0.0,0.0)
	\pspolygon[dimen=middle, linewidth=0.01, linecolor=lightgray, fillstyle=solid, fillcolor=lightgray](4.0,3.5)(4.0,0.0)(2.5,0.0)(1.5,1.0)	
	
	\psset{linestyle=dotted}
	
	\psline(0.0,0.0)(2.0,2.0)(4.0,0.0)
	\psline(0.0,2.0)(1.0,1.0)
	\psline(4.0,2.0)(3.0,1.0)
	\psline(0.0,4.0)(1.0,1.0)
	\psline(4.0,4.0)(3.0,1.0)
	\psline(0.0,4.0)(2.0,2.0)
	\psline(4.0,4.0)(2.0,2.0)
	
	\psset{linestyle=solid}

	\psaxes[dx=2, Dx=0.5, dy=2, Dy=0.5](0.0,0.0)(0.0,0.0)(4.0,4.0)
	\psdot(1.5,1.0)
	\psline(0.5,0.0)(4.0,3.5)
	\psline(0.0,2.5)(2.5,0.0)
	
	\psdots[dotstyle=Bo, fillcolor=white](0.5,0.0)(4.0,3.5)
\end{pspicture}
			}
		}
		\subfigure[$(z, G(z)) = (\frac{1}{8}, \frac{1}{4})$, and $\widehat{P} = \tfrac{1}{5}$]{
			\scalebox{0.7}{
				\begin{pspicture}(-0.5,-0.5)(4.5,4.0)
	\psset{linewidth=0.04}
	\psset{dotsize=0 6}
	
	\pspolygon[dimen=middle, linewidth=0.01, linecolor=lightgray, fillstyle=solid, fillcolor=lightgray](0.0,1.5)(0.0,0.5)(0.5,1.0)
	\pspolygon[dimen=middle, linewidth=0.01, linecolor=lightgray, fillstyle=solid, fillcolor=lightgray](0.5,1.0)(1.5,0.0)(4.0,0.0)(4.0,4.0)(3.5,4.0)
	
	\psset{linestyle=dotted}
	
	\psline(0.0,0.0)(2.0,2.0)(4.0,0.0)
	\psline(0.0,2.0)(1.0,1.0)
	\psline(4.0,2.0)(3.0,1.0)
	\psline(0.0,4.0)(1.0,1.0)
	\psline(4.0,4.0)(3.0,1.0)
	\psline(0.0,4.0)(2.0,2.0)
	\psline(4.0,4.0)(2.0,2.0)
	
	\psset{linestyle=solid}
	
	\psaxes[dx=2, Dx=0.5, dy=2, Dy=0.5](0.0,0.0)(0.0,0.0)(4.0,4.0)
	\psdot(0.5,1.0)
	\psline(0.0,0.5)(3.5,4.0)
	\psline(0.0,1.5)(1.5,0.0)
	
	\psdots[dotstyle=Bo, fillcolor=white](1.5,0.0)(4.0,2.5)
\end{pspicture}
			}
		}
		\subfigure[$(z, G(z)) = (\frac{1}{8}, \frac{1}{2})$, and $\widehat{P} = \tfrac{3}{11}$]{
			\scalebox{0.7}{
				\begin{pspicture}(-0.5,-0.5)(4.5,4.0)
	\psset{linewidth=0.04}
	\psset{dotsize=0 6}
	
	\pspolygon[dimen=middle, linewidth=0.01, linecolor=lightgray, fillstyle=solid, fillcolor=lightgray](0.0,2.5)(0.0,1.5)(0.5,2.0)
	\pspolygon[dimen=middle, linewidth=0.01, linecolor=lightgray, fillstyle=solid, fillcolor=lightgray](0.5,2.0)(2.5,0.0)(4.0,0.0)(4.0,4.0)(2.5,4.0)
	
	\psset{linestyle=dotted}
	
	\psline(0.0,0.0)(2.0,2.0)(4.0,0.0)
	\psline(0.0,2.0)(1.0,1.0)
	\psline(4.0,2.0)(3.0,1.0)
	\psline(0.0,4.0)(1.0,1.0)
	\psline(4.0,4.0)(3.0,1.0)
	\psline(0.0,4.0)(2.0,2.0)
	\psline(4.0,4.0)(2.0,2.0)
	
	\psset{linestyle=solid}

	\psaxes[dx=2, Dx=0.5, dy=2, Dy=0.5](0.0,0.0)(0.0,0.0)(4.0,4.0)
	\psdot(0.5,2.0)
	\psline(0.0,1.5)(2.5,4.0)
	\psline(0.0,2.5)(2.5,0.0)
	
	\psdots[dotstyle=Bo, fillcolor=white](4.0,0.0)(1.25,2.75)
\end{pspicture}
			}
		}
		\subfigure[$(z, G(z)) = (\frac{1}{4}, \frac{1}{2})$, and $\widehat{P} = \tfrac{1}{3}$]{
			\scalebox{0.7}{
				\begin{pspicture}(-0.5,-0.5)(4.5,4.0)
	\psset{linewidth=0.04}
	\psset{dotsize=0 6}
	
	\pspolygon[dimen=middle, linewidth=0.01, linecolor=lightgray, fillstyle=solid, fillcolor=lightgray](0.0,3.0)(0.0,1.0)(1.0,2.0)
	\pspolygon[dimen=middle, linewidth=0.01, linecolor=lightgray, fillstyle=solid, fillcolor=lightgray](1.0,2.0)(3.0,0.0)(4.0,0.0)(4.0,4.0)(3.0,4.0)
	
	\psset{linestyle=dotted}
	
	\psline(0.0,0.0)(2.0,2.0)(4.0,0.0)
	\psline(0.0,2.0)(1.0,1.0)
	\psline(4.0,2.0)(3.0,1.0)
	\psline(0.0,4.0)(1.0,1.0)
	\psline(4.0,4.0)(3.0,1.0)
	\psline(0.0,4.0)(2.0,2.0)
	\psline(4.0,4.0)(2.0,2.0)
	
	\psset{linestyle=solid}

	\psaxes[dx=2, Dx=0.5, dy=2, Dy=0.5](0.0,0.0)(0.0,0.0)(4.0,4.0)
	\psdot(1.0,2.0)
	\psline(0.0,1.0)(3.0,4.0)
	\psline(0.0,3.0)(3.0,0.0)
	
	\psdots[dotstyle=Bo, fillcolor=white](3.0,0.0)(0.0,3.0)
\end{pspicture}
			}
		}
		\subfigure[$(z, G(z)) = (\frac{3}{8}, \frac{7}{8})$, and $\widehat{P} = 0$]{
			\scalebox{0.7}{
				\begin{pspicture}(-0.5,-0.5)(4.5,4.0)
	\psset{linewidth=0.04}
	\psset{dotsize=0 6}
	
	\pspolygon[dimen=middle, linewidth=0.01, linecolor=lightgray, fillstyle=solid, fillcolor=lightgray](0.0,2.0)(1.5,3.5)(1.0,4.0)(0.0,4.0)
	\pspolygon[dimen=middle, linewidth=0.01, linecolor=lightgray, fillstyle=solid, fillcolor=lightgray](1.5,3.5)(4.0,1.0)(4.0,4.0)(2.0,4.0)(2.0,4.0)
	
	\psset{linestyle=dotted}
	
	\psline(0.0,0.0)(2.0,2.0)(4.0,0.0)
	\psline(0.0,2.0)(1.0,1.0)
	\psline(4.0,2.0)(3.0,1.0)
	\psline(0.0,4.0)(1.0,1.0)
	\psline(4.0,4.0)(3.0,1.0)
	\psline(0.0,4.0)(2.0,2.0)
	\psline(4.0,4.0)(2.0,2.0)
	
	\psset{linestyle=solid}

	\psaxes[dx=2, Dx=0.5, dy=2, Dy=0.5](0.0,0.0)(0.0,0.0)(4.0,4.0)
	\psdot(1.5,3.5)
	\psline(0.0,2.0)(2.0,4.0)
	\psline(1.0,4.0)(4.0,1.0)
\end{pspicture}
			}
		}
		\caption{Illustration of the maximizers in Example \ref{eg:Phat_1dim_example} with $L = 1$.  The dotted lines show the boundaries of the various cases in $(z, G(z))$ data space.  The black dot shows the data point, and the white dots the positions of $(x_{0}, y_{0})$ and $(x_{1}, y_{1})$ that maximize the probability of failure;  in each case, $\widehat{P} = \big( 1 - \frac{m_{+}}{| y_{1} - y_{0} |} \big)_{+}$.  Note that failure is impossible in case (e).}
		\label{fig:discty}
	\end{figure}
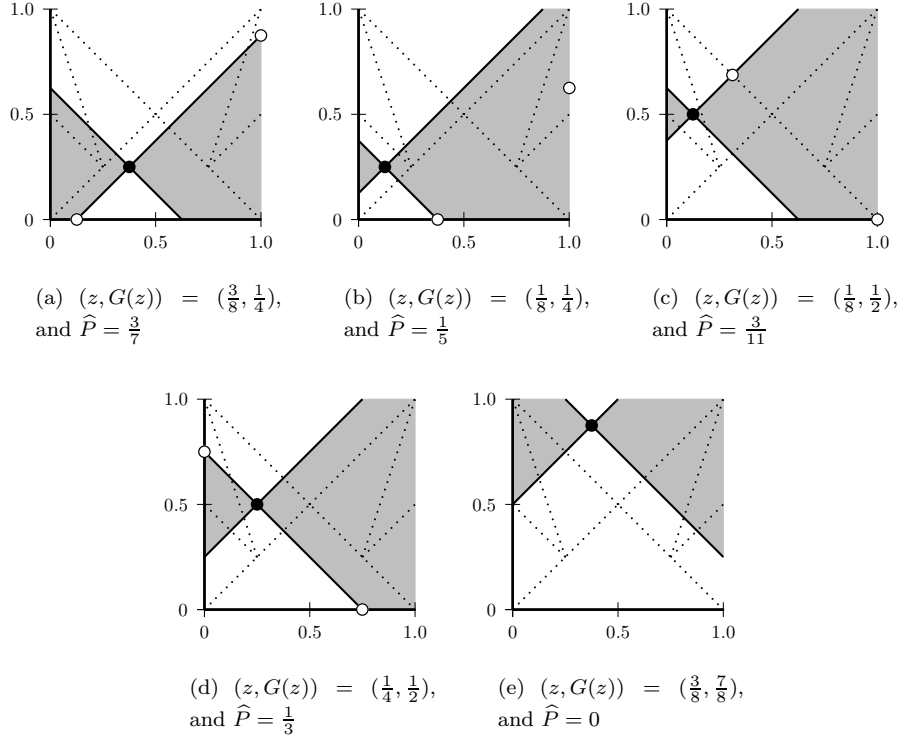
	
	Note that, for any single observation $(z, G(z))$, the least upper bound on the McDiarmid diameter, $\widehat{D}[G]$, is simply $L$, and that the bound \eqref{eq:discontinuity} is in each case an improvement on both McDiarmid's inequality
	\[
		\P[G(X) \leq 0] \leq \exp \left( -2 m_{+}^{2} \middle/ \widehat{D}[G]^{2} \right) = \exp \left( - 2 m_{+}^{2} \middle/ L^{2} \right)
	\]
	and on the $K = 1$ optimal McDiarmid inequality \cite[\S4]{OwhadiScovelSullivanMcKernsOrtiz:2010}
	\[
		\P[G(X) \leq 0] \leq \left( 1 - \frac{m_{+}}{\widehat{D}[G]} \right)_{+} = \left( 1 - \frac{m_{+}}{L} \right)_{+}.
	\]
\end{eg}

\section{Redundant and Non-Binding Observations}
\label{sec:redundancy}

In many applications, the aim is not to understand the behaviour of $G$ on the whole of the input parameter space $\mathcal{X}$, but only on some subset $V \subseteq \mathcal{X}$, or on the elements of a partition $\mathcal{X} = \biguplus_{j = 1}^{J} V_{j}$ of $\mathcal{X}$ \cite{SullivanTopcuMcKernsOwhadi:2011}.  The observation set $\mathcal{O}$ may lie entirely within $V$, or only partially lie within $V$, or lie entirely outside $V$.  Heuristically, it seems reasonable that the points of $\mathcal{O}$ that are ``nearest'' to $V$ should be the most important ones, but it is not immediately obvious what ``nearest'' means.  

However, the formulation of the UQ objectives as optimization problems provides a natural notion of information content.  Instead of calculating, for example,  information-theoretic entropies, we simply make use of notions of relevancy that are natural to the optimization-theoretic context:  the relevant data points are the precisely the ones that correspond to non-trivial constraints, or rather, determine the extreme value of the optimization problem.

Even if the aim is to understand the behaviour of $G$ on all of $\mathcal{X}$ rather than a subset $V \subseteq \mathcal{X}$, the problems \eqref{eq:opt_diam_1} and \eqref{eq:opt_prob}--\eqref{eq:opt_prob_feasible} are highly constrained, and their solution is much simplified by elimination of redundant constraints/observations.  To that end, this section discusses two notions of redundancy/relevancy for data points and other constraints \cite{Holder:2006}:
\begin{itemize}
	\item \emph{redundant} constraints do not change the feasible set in the problems \eqref{eq:opt_diam_1} and \eqref{eq:opt_prob}--\eqref{eq:opt_prob_feasible};
	\item \emph{non-binding} constraints may change the feasible set in the problems \eqref{eq:opt_diam_1} and \eqref{eq:opt_prob}--\eqref{eq:opt_prob_feasible}, and may even change the extremizer, but do not change the extreme value.
\end{itemize}
Clearly, every redundant constraint is non-binding, but not \emph{vice versa}.  With this point of view, the problem of finding ``nearest data points'' becomes one of finding minimal data sets $\mathcal{O}$ that are \emph{redundancy-free}.

\subsection{Redundant Lipschitz Constraints}

In problem \eqref{eq:opt_prob_reduced}, many of the $2^{2 K}$ Lipschitz constraints of the form
\begin{equation}
	\label{eq:shortness_cons}
	| y_{\eps} - y_{\eps'} | \leq d_{L}(x_{\eps}, x_{\eps'})
\end{equation}
are redundant constraints.  First, \eqref{eq:shortness_cons} is obviously satisfied when $\eps = \eps'$, so there are at most $2^{2 K} - 2^{K} = 2^{K} (2^{K} - 1)$ non-redundant constraints of the form \eqref{eq:shortness_cons}.  Secondly, \eqref{eq:shortness_cons} is symmetric under interchange of $\eps$ and $\eps'$, and so there are at most $2^{K} (2^{K} - 1) / 2 = 2^{K - 1} (2^{K} - 1)$ non-redundant constraints of the form \eqref{eq:shortness_cons};  it suffices to endow $\{ 0, 1 \}^{K}$ with some total order $\preceq$ (\emph{e.g.}\ lexicographic ordering) and only verify \eqref{eq:shortness_cons} for $\eps \prec \eps'$.

A third source of redundancy is neatly encapsulated in Lemma \ref{lem:Lipschitz_short}:  in order to verify that \eqref{eq:shortness_cons} holds for all $\eps, \eps' \in \{ 0, 1 \}^{K}$ (\emph{i.e.}\ to show that $g|_{\mathcal{C}(x_{0}, x_{1})}$ is $d_{L}$-short, where $y_{\eps} = g(x_{\eps})$), it is necessary and sufficient to check that \eqref{eq:shortness_cons} holds when $\eps$ and $\eps'$ differ in precisely one entry.  Geometrically, this corresponds to checking \eqref{eq:shortness_cons} not between arbitrary vertices of the cube $\mathcal{C}(x_{0}, x_{1})$ but only along edges joining adjacent vertices.  There are $K 2^{K}$ such edges, and so symmetry considerations yield the following result:

\begin{thm}[Relevant Lipschitz constraints]
	\label{thm:relevant_Lipschitz_cons}
	A constraint of the form \eqref{eq:shortness_cons} in problem \eqref{eq:opt_prob_reduced} is relevant only if $\eps \prec \eps'$ and $\eps_{k} \neq \eps'_{k}$ for precisely one $k \in \{ 1, \dots, K \}$;  otherwise, it is redundant.  Hence, there are at most $K 2^{K - 1}$ non-redundant constraints of the form \eqref{eq:shortness_cons}.
\end{thm}

\subsection{Redundant Data Points}

Given $V \subseteq \mathcal{X}$ and $\mathcal{O} \subseteq \mathcal{X}$ such that $G|_{\mathcal{O}}$ is known, an observation $(z_{0}, G(z_{0})) \in \mathcal{X} \times \R$ is said to be \emph{redundant on $V$ with respect to $\mathcal{O}$} if, for all $(x, y) \in V \times \R$,
\begin{equation}
	\label{eq:redundancy}
	\left. \begin{matrix}
		\text{for all } z \in \mathcal{O}, \\
		| y - G(z) | \leq d_{L}(x, z)
	\end{matrix} \right\}
	\implies
	| y - G(z_{0}) | \leq d_{L}(x, z_{0}),
\end{equation}
and say that it is \emph{relevant} otherwise.  That is, a redundant observation is one for which the induced constraint in \eqref{eq:opt_diam_1} (or \eqref{eq:opt_prob}--\eqref{eq:opt_prob_feasible} or \eqref{eq:opt_prob_reduced}) is automatically satisfied whenever the constraints induced by $\mathcal{O}$ are satisfied;  put another way, the set of $G|_{\mathcal{O}}$-feasible points in $V \times \R$ is contained in the cone of $G|_{\{ z_{0} \}}$-feasible points in $V \times \R$.  See Figure \ref{fig:redundancy} for an illustration.

\begin{figure}
	\scalebox{0.7}{
		\begin{pspicture}(-5.5,-3)(5.5,3)
	\psset{linewidth=0.04}
	\psset{arrowsize=0 6}
	\psset{dotsize=0 6}
	\psset{shadowcolor=lightgray}
	\psset{shadowsize=2pt}
	\psset{shadowangle=-30}
	
	\pspolygon[dimen=middle, linecolor=lightgray, fillstyle=solid, fillcolor=lightgray](-4.5,1.5)(-4.5,0.5)(-4,1)
	\pspolygon[dimen=middle, linecolor=lightgray, fillstyle=solid, fillcolor=lightgray]
	(-4.0,1.0)(-0.75,-2.25)(0.5,-1.0)(-2.75,2.25)
	\pspolygon[dimen=middle, linecolor=lightgray, fillstyle=solid, fillcolor=lightgray](0.5,-1)(1.5,-2)(1.5,0)

	\psdot[linecolor=gray](2.5,-0.5)
	\rput(2.5,-2.0){$(z'_{0}, G(z'_{0}))$}
	\pscurve[shadow=true]{->}(2.5,-1.75)(2.4,-1.25)(2.6,-1.25)(2.5,-0.65)
	\psline[linecolor=gray](0.6,-2.4)(4.5,1.5)
	\psline[linecolor=gray](1.0,1.0)(3.6,-1.6)

	\psdot[linecolor=gray](4.5,-0.5)
	\rput(4.5,-2.0){$(z_{0}, G(z_{0}))$}
	\pscurve[shadow=true]{->}(4.5,-1.75)(4.4,-1.25)(4.6,-1.25)(4.5,-0.65)
	\psline[linecolor=gray](3.4,-1.6)(4.5,-0.5)
	\psline[linecolor=gray](1.6,2.4)(4.5,-0.5)

	\psdot(-4,1)
	\rput(-4,-1.0){$(z_{1}, G(z_{1}))$}
	\pscurve[shadow=true]{->}(-4.0,-0.75)(-3.8,0.0)(-4.2,0.0)(-4.05,0.8)
	\psline(-4.5,1.5)(-0.6,-2.4)
	\psline(-4.5,0.5)(-2.6,2.4)
	
	\psdot(0.5,-1)
	\rput(0.5,1.5){$(z_{2}, G(z_{2}))$}
	\pscurve[shadow=true]{->}(0.5,1.3)(0.7,0.2)(0.3,0.2)(0.45,-0.8)
	\psline(-0.9,-2.4)(3.9,2.4)
	\psline(-2.9,2.4)(1.9,-2.4)	

	\psaxes[labels=none, ticks=none](-5,-2.5)(-5,-2.5)(5,2.5)
	\psline(-4.5,-2.6)(-4.5,-2.4)
	\psline(1.5,-2.6)(1.5,-2.4)
	\psline(4.5,-2.6)(4.5,-2.4)
	\rput[r](-5.25,0){$\R$}
	\rput[t](-1.5,-2.75){$V$}
	\rput[t](3.0,-2.75){$\mathcal{X} \setminus V$}
\end{pspicture}
	}
	\caption{The observation at $z_{0} \in \mathcal{X} \setminus V$ is redundant on $V$ with respect to $\mathcal{O} := \{ z_{1}, z_{2} \}$, since its feasible cone contains the set of all $G|_{\mathcal{O}}$-feasible points in $V \times \R$.  Contrarily, the observation at $z'_{0}$ is relevant on $V$ with respect to $\mathcal{O}$.}
	\label{fig:redundancy}
\end{figure}
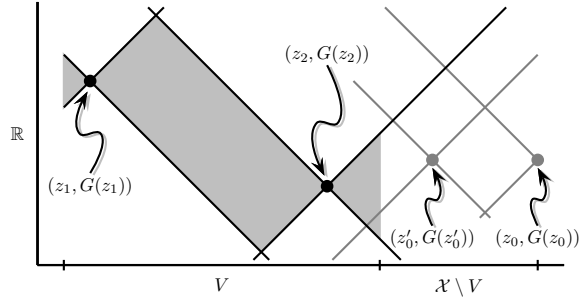

Proposition \ref{prop:relevant_data} shows that every (non-isolated) data point $z \in \mathcal{O} \cap V$ is relevant;  only data points $z \in \mathcal{O} \setminus V$ may be redundant.  Furthermore, Theorem \ref{thm:redundant_data} shows that every point $z \in \mathcal{O} \setminus V$ that is sufficiently far away from $V$ is redundant.

\begin{prop}[Relevant data points]
	\label{prop:relevant_data}
	Let $V \subseteq \mathcal{X}$, $\mathcal{O} \subseteq \mathcal{X}$, and $G|_{\mathcal{O}}$ be given, and suppose that $d_{L}$ is a metric.  If $z_{0} \in \mathcal{O} \cap V$ and $z_{0}$ is an isolated point of $\mathcal{O}$, then $z_{0}$ is relevant on $V$ with respect to $\mathcal{O} \setminus \{ z_{0} \}$.
\end{prop}

\begin{proof}
	Let $z'$ be the closest point of $\mathcal{O} \setminus \{ z_{0} \}$ to $z_{0}$ (if there is more than one such point, then choose any such point).  Then any value
	\[
		y \in [G(z') - d_{L}(z_{0}, z'), G(z') + d_{L}(z_{0}, z')]
	\]
	is feasible with respect to $\mathcal{O} \setminus \{ z_{0} \}$.  Since $z_{0}$ is an isolated point of $\mathcal{O}$ and $d_{L}$ is a metric, this interval has non-zero length.  However, the such $y$ that is feasible with respect to $\mathcal{O}$ is $G(z_{0})$.  Hence, $z_{0}$ supplies a non-trivial constraint and is relevant on $V$ with respect to $\mathcal{O} \setminus \{ z_{0} \}$.  (Note that if $V$ is, say, a subset of $\R^{K}$ with non-empty interior, then this argument can be applied on a neighbourhood of $z_{0}$, thereby demonstrating relevancy of $z_{0}$ to a non-trivial set.)
\end{proof}

The next result gives a sufficient condition for observations $z_{0} \in \mathcal{O} \setminus V$ to be redundant.  Say that $y \in \mathcal{X}$ is \emph{between $x \in \mathcal{X}$ and $z \in \mathcal{X}$} if
\begin{equation}
	\label{eq:between}
	d_{L}(x, z) = d_{L}(x, y) + d_{L}(y, z),
\end{equation}
and that $y$ is \emph{between $V \subseteq \mathcal{X}$ and $W \subseteq \mathcal{X}$} if \eqref{eq:between} holds for every $x \in V$ and $z \in W$.  Note well that in the prototypical case that $d_{L}$ is the $\ell^{1}$ Manhattan metric on $\R^{K}$, the set of points between $x$ and $z$ is not the Euclidean line segment joining them, but the closed convex hull $\overline{\mathrm{co}}(\mathcal{C}(x, z))$, \emph{i.e.}\ the compact cuboid with faces perpendicular to the coordinate axes and $x$ and $z$ as its opposite corners.

\begin{thm}[Redundant data points]
	\label{thm:redundant_data}
	Let $V \subseteq \mathcal{X}$, $\mathcal{O} \subseteq \mathcal{X}$, $L$ and $G|_{\mathcal{O}}$ be given.  Fix $z_{0} \in \mathcal{O} \setminus V$.  Suppose that $p \in \mathcal{X}$ is between $V$ and $z_{0}$, and that there exist $z', z'' \in \mathcal{O} \cap V$ satisfying
	\begin{align}
		\label{eq:redundancy_cond_1}
		G(z') + d_{L}(z', p) & \leq G(z_{0}) + d_{L}(z_{0}, p), \\
		\label{eq:redundancy_cond_2}
		G(z'') - d_{L}(z'', p) & \geq G(z_{0}) - d_{L}(z_{0}, p).
	\end{align}
	Then $z_{0}$ is redundant on $V$ with respect to $\mathcal{O} \cap V$.
\end{thm}

\begin{proof}
	Let $(x, y) \in V \times \R$ be a feasible point with respect to $G|_{\mathcal{O}\cap V}$, \emph{i.e.}\ 
	\[
		| y - G(z) | \leq d_{L}(x, z) \text{ for each } z \in \mathcal{O} \cap V,
	\]
	and suppose for a contradiction that $| y - G(z_{0}) | > d_{L}(x, z_{0}) > 0$.  If $y > G(z_{0})$, then the assumption \emph{ad absurdum} implies that $y > G(z_{0}) + d_{L}(x, z_{0})$.  Hence,
	\begin{align*}
		| y - G(z') | 
		& \geq y - G(z') \\
		& > G(z_{0}) + d_{L}(x, z_{0}) - G(z') \\
		& \geq d_{L}(z', p) - d_{L}(z_{0}, p) + d_{L}(x, z_{0}) && \text{by \eqref{eq:redundancy_cond_1}} \\
		& = d_{L}(z', p) + d_{L}(x, p) && \text{since $p$ is between $V$ and $z_{0}$} \\
		& \geq d_{L}(x, z') && \text{by the triangle inequality,}
	\end{align*}
	which contradicts the feasibility of $(x, y)$ with respect to $G|_{\mathcal{O}\cap V}$.  Similarly, if $y < G(z_{0})$, then \eqref{eq:redundancy_cond_2} implies that
	\[
		| y - G(z'') | > d_{L}(x, z''),
	\]
	which is again a contradiction.  This completes the proof.
\end{proof}

If the closure $\overline{V}$ of $V$ is a compact rectangular box $\prod_{k = 1}^{K} [\alpha^{k}, \beta^{k}] \subseteq \R^{K}$, then, for each $z_{0} \in \mathcal{O} \setminus V$, there is a natural choice for the point $p$ with respect to which conditions \eqref{eq:redundancy_cond_1} and \eqref{eq:redundancy_cond_2} can be checked:  the unique point $P_{z_{0}, V} \in \overline{V}$ that is closest to $z_{0}$, where
\begin{equation}
	\label{eq:proj_V}
	P_{x, V}^{k} := \begin{cases}
		\alpha^{k}, & \text{ if $x^{k} < \alpha^{k}$,} \\
		x^{k}, & \text{ if $\alpha^{k} \leq x^{k} \leq \beta^{k}$,} \\
		\beta^{k}, & \text{ if $x^{k} > \beta^{k}$.}
	\end{cases}
\end{equation}
It is easy to see that $P_{z_{0}, V}$ is between $V$ and $z_{0}$.  This choice of $p$ validates the heuristic that observations far away from $V$ ought to be redundant, since \eqref{eq:redundancy_cond_1} and \eqref{eq:redundancy_cond_2} are certain to hold when $V$ is bounded and $d_{L}(z_{0}, V)$ is large enough.

\subsection{Non-Binding Data Points}

A more interesting notion of the information content of the data points $(z, G(z))$ is not relevancy but \emph{bindingness}.  Whereas redundancy concerns the set of feasible points for an optimization problem, a \emph{non-binding} constraint (or data point) is one that perhaps changes the feasible set but does not change the extreme value of the problem.

Given $\mathcal{O} \subseteq \mathcal{X}$ such that $G|_{\mathcal{O}}$ is known, an observation $(z_{0}, G(z_{0})) \in \mathcal{X} \times \R$ is said to be 
\begin{itemize}
	\item \emph{non-binding for $\widehat{D}_{k}$ with respect to $\mathcal{O}$} if
	\[
		\widehat{D}_{k}[\mathcal{X}, G|_{\mathcal{O} \cup \{ z_{0} \}}, L] = \widehat{D}_{k}[\mathcal{X}, G|_{\mathcal{O}}, d_{L}];
	\]
	\item \emph{non-binding for $\widehat{P}$ with respect to $\mathcal{O}$} if
	\[
		\widehat{P}[\mathcal{X}, G|_{\mathcal{O} \cup \{ z_{0} \}}, L, m] = \widehat{P}[\mathcal{X}, G|_{\mathcal{O}}, L, m].
	\]
\end{itemize}	
Otherwise, an observation is said to be \emph{binding}.  Note well that the inclusion of a binding observation \emph{strictly} changes the extreme value of the optimization problems, not just the set of extremizers.

Clearly, if including an observation at $z_{0}$ does not change the feasible set for, say, the $\widehat{D}_{k}$ problem \eqref{eq:opt_diam_1}, then including it does not change the extreme value of \eqref{eq:opt_diam_1};  that is, every redundant data point is non-binding, and every binding data point is relevant.  The converse implications, however, are false:  in general, there are data points that do change the feasible set for the optimization problems for $\widehat{D}_{k}$ and $\widehat{P}$, but do not change the extreme values.  See Figure \ref{fig:non-binding} for some illustrations based upon the earlier Example \ref{eg:Phat_1dim_example}.  See also Figure \ref{fig:second}, which illustrates the set of all second data points $(z_{2}, G(z_{2})) \in [0, 1] \times \R$ that are redundant with respect to the first data point from Example \ref{eg:Phat_1dim_example}.

\begin{figure}
	\subfigure[(Non-unique) maximizer for the probability of failure with one data point at $(\frac{3}{8}, \frac{1}{8})$.]{
		\scalebox{0.7}{
			\begin{pspicture}(-0.5,-0.5)(4.5,4.0)
	\psset{linewidth=0.04}
	\psset{dotsize=0 6}
	
	\pspolygon[dimen=middle, linewidth=0.01, linecolor=lightgray, fillstyle=solid, fillcolor=lightgray](0.0,2.0)(1.5,0.5)(1.0,0.0)(0.0,0.0)
	\pspolygon[dimen=middle, linewidth=0.01, linecolor=lightgray, fillstyle=solid, fillcolor=lightgray](4.0,3.0)(4.0,0.0)(2.0,0.0)(1.5,0.5)	
	
	\psset{linestyle=dotted}
	
	\psline(0.0,0.0)(2.0,2.0)(4.0,0.0)
	\psline(0.0,2.0)(1.0,1.0)
	\psline(4.0,2.0)(3.0,1.0)
	\psline(0.0,4.0)(1.0,1.0)
	\psline(4.0,4.0)(3.0,1.0)
	\psline(0.0,4.0)(2.0,2.0)
	\psline(4.0,4.0)(2.0,2.0)
	
	\psset{linestyle=solid}

	\psaxes[dx=2, Dx=0.5, dy=2, Dy=0.5](0.0,0.0)(0.0,0.0)(4.0,4.0)
	\psdot(1.5,0.5)
	\psline(1.0,0.0)(4.0,3.0)
	\psline(0.0,2.0)(2.0,0.0)
	
	\psdots[dotstyle=Bo, fillcolor=white](1.0,0.0)(4.0,3.0)
\end{pspicture}
		}
	}
	\subfigure[A non-binding new data point;  the maximizer does not change.  \emph{Cf.}\ Figure \ref{fig:second}(a).]{
		\scalebox{0.7}{
			\begin{pspicture}(-0.5,-0.5)(4.5,4.0)
	\psset{linewidth=0.04}
	\psset{dotsize=0 6}
	
	\pspolygon[dimen=middle, linewidth=0.01, linecolor=lightgray, fillstyle=solid, fillcolor=lightgray](0.0,0.5)(0.5,0.0)(1.0,0.0)(1.5,0.5)(0.75,1.25)
	\pspolygon[dimen=middle, linewidth=0.01, linecolor=lightgray, fillstyle=solid, fillcolor=lightgray](4.0,3.0)(4.0,0.0)(2.0,0.0)(1.5,0.5)	
	
	\psaxes[dx=2, Dx=0.5, dy=2, Dy=0.5](0.0,0.0)(0.0,0.0)(4.0,4.0)
	\psdot(1.5,0.5)
	\psline(1.0,0.0)(4.0,3.0)
	\psline(0.0,2.0)(2.0,0.0)
	
	\psdot(0.0,0.5)
	\psline(0.0,0.5)(0.5,0.0)
	\psline(0.0,0.5)(3.5,4.0)
	
	\psdots[dotstyle=Bo, fillcolor=white](1.0,0.0)(4.0,3.0)
\end{pspicture}
		}
	}
	\subfigure[A non-binding new data point;  the maximizer changes but the maximum value does not.]{
		\scalebox{0.7}{
			\begin{pspicture}(-0.5,-0.5)(4.5,4.0)
	\psset{linewidth=0.04}
	\psset{dotsize=0 6}
	
	\pspolygon[dimen=middle, linewidth=0.01, linecolor=lightgray, fillstyle=solid, fillcolor=lightgray](0.0,1.25)(0.75,0.5)(0.25,0.0)(0.0,0.0)
	\pspolygon[dimen=middle, linewidth=0.01, linecolor=lightgray, fillstyle=solid, fillcolor=lightgray](0.75,0.5)(1.125,0.125)(1.5,0.5)(1.125,0.875)
	\pspolygon[dimen=middle, linewidth=0.01, linecolor=lightgray, fillstyle=solid, fillcolor=lightgray](4.0,3.0)(4.0,0.0)(2.0,0.0)(1.5,0.5)	
	
	\psaxes[dx=2, Dx=0.5, dy=2, Dy=0.5](0.0,0.0)(0.0,0.0)(4.0,4.0)
	\psdot(1.5,0.5)
	\psline(1.0,0.0)(4.0,3.0)
	\psline(0.0,2.0)(2.0,0.0)
	
	\psdot(0.75,0.5)
	\psline(0.0,1.25)(1.25,0.0)
	\psline(0.25,0.0)(4.0,3.75)
	
	\psdots[dotstyle=Bo, fillcolor=white](0.25,0.0)(4.0,3.0)
\end{pspicture}
		}
	}
	\subfigure[A binding new data point:  the maximizer and maximum value both change.]{
		\scalebox{0.7}{
			\begin{pspicture}(-0.5,-0.5)(4.5,4.0)
	\psset{linewidth=0.04}
	\psset{dotsize=0 6}
	
	\pspolygon[dimen=middle, linewidth=0.01, linecolor=lightgray, fillstyle=solid, fillcolor=lightgray](0.0,2.0)(1.5,0.5)(1.0,0.0)(0.0,0.0)
	\pspolygon[dimen=middle, linewidth=0.01, linecolor=lightgray, fillstyle=solid, fillcolor=lightgray](1.5,0.5)(2.0,0.0)(2.5,0.5)(2.0,1.0)
	\pspolygon[dimen=middle, linewidth=0.01, linecolor=lightgray, fillstyle=solid, fillcolor=lightgray](2.5,0.5)(3.0,0.0)(4.0,0.0)(4.0,2.0)
	
	\psaxes[dx=2, Dx=0.5, dy=2, Dy=0.5](0.0,0.0)(0.0,0.0)(4.0,4.0)
	\psdot(1.5,0.5)
	\psline(1.0,0.0)(4.0,3.0)
	\psline(0.0,2.0)(2.0,0.0)
	
	\psdot(2.5,0.5)
	\psline(2.0,0.0)(4.0,2.0)
	\psline(3.0,0.0)(0.0,3.0)
	
	\psdots[dotstyle=Bo, fillcolor=white](1.0,0.0)(4.0,2.0)
\end{pspicture}
		}
	}
	\subfigure[Two binding new data points that together render failure impossible.]{
		\scalebox{0.7}{
			\begin{pspicture}(-0.5,-0.5)(4.5,4.0)
	\psset{linewidth=0.04}
	\psset{dotsize=0 6}
	
	\pspolygon[dimen=middle, linewidth=0.01, linecolor=lightgray, fillstyle=solid, fillcolor=lightgray](0.0,1.5)(0.0,0.5)(0.5,1.0)
	\pspolygon[dimen=middle, linewidth=0.01, linecolor=lightgray, fillstyle=solid, fillcolor=lightgray](0.5,1.0)(1.25,0.25)(1.5,0.5)(0.75,1.25)
	\pspolygon[dimen=middle, linewidth=0.01, linecolor=lightgray, fillstyle=solid, fillcolor=lightgray](1.5,0.5)(1.75,0.25)(3.0,1.5)(2.75,1.75)
	\pspolygon[dimen=middle, linewidth=0.01, linecolor=lightgray, fillstyle=solid, fillcolor=lightgray](3.0,1.5)(4.0,0.5)(4.0,2.5)
	
	\psaxes[dx=2, Dx=0.5, dy=2, Dy=0.5](0.0,0.0)(0.0,0.0)(4.0,4.0)
	\psdot(1.5,0.5)
	\psline(1.0,0.0)(4.0,3.0)
	\psline(0.0,2.0)(2.0,0.0)
	
	\psdot(0.5,1.0)
	\psline(0.0,1.5)(1.5,0.0)
	\psline(0.0,0.5)(3.5,4.0)

	\psdot(3.0,1.5)
	\psline(1.5,0.0)(4.0,2.5)
	\psline(0.5,4.0)(4.0,0.5)
\end{pspicture}
		}
	}
	\caption{Additional binding and non-binding data points for the one-dimensional Example \ref{eg:Phat_1dim_example}.  As before, black dots show data points and white dots the locations of maximizing $(x_{0}, y_{0})$ and $(x_{1}, y_{1})$, with $\widehat{P} = \big( 1 - \frac{m_{+}}{| y_{1} - y_{0} |} \big)_{+}$.}
	\label{fig:non-binding}
\end{figure}
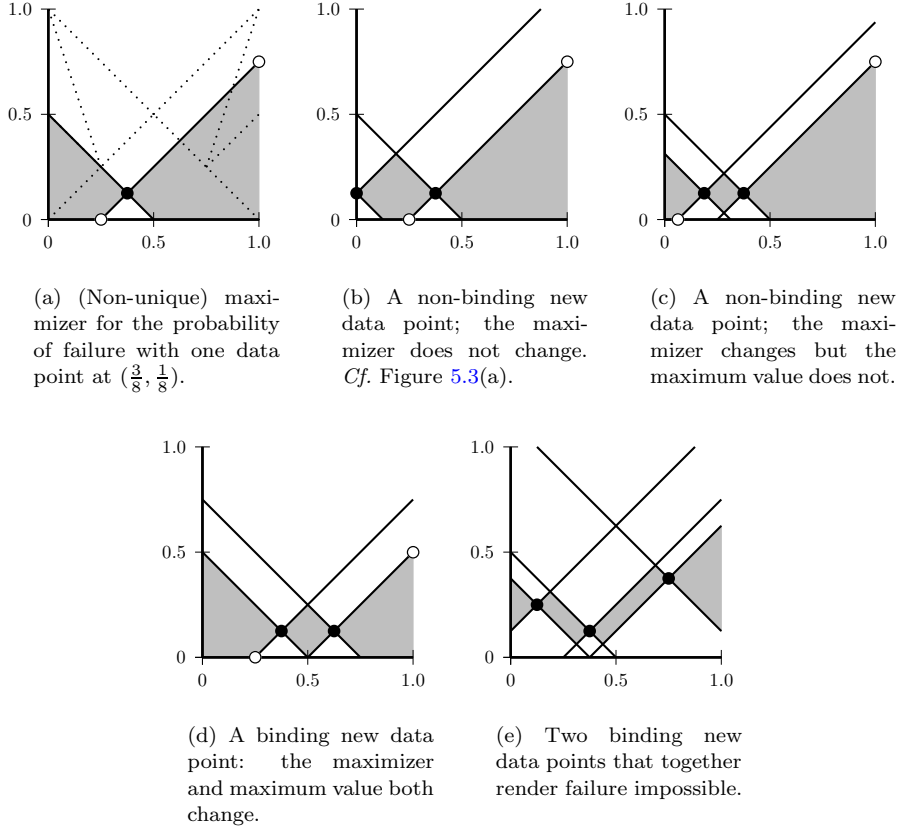

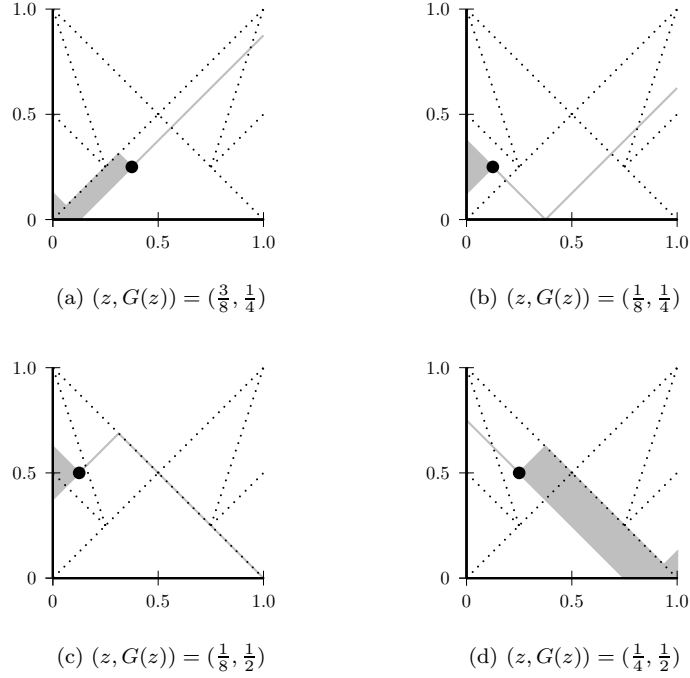
\begin{figure}
		\subfigure[$(z, G(z)) = (\frac{3}{8}, \frac{1}{4})$]{
			\scalebox{0.7}{
				\begin{pspicture}(-1.5,-0.5)(5.5,4.0)
	\psset{linewidth=0.04}
	\psset{dotsize=0 6}
	
		
	\pspolygon[dimen=middle, linecolor=lightgray, fillstyle=solid, fillcolor=lightgray](0.0,0.5)(0.0,0.0)(0.5,0.0)(1.5,1.0)(1.25,1.25)(0.25,0.25)
	\psline[linecolor=lightgray](0.5,0.0)(4.0,3.5)
	
	\psset{linestyle=dotted}
	
	\psline(0.0,0.0)(2.0,2.0)(4.0,0.0)
	\psline(0.0,2.0)(1.0,1.0)
	\psline(4.0,2.0)(3.0,1.0)
	\psline(0.0,4.0)(1.0,1.0)
	\psline(4.0,4.0)(3.0,1.0)
	\psline(0.0,4.0)(2.0,2.0)
	\psline(4.0,4.0)(2.0,2.0)
	
	\psset{linestyle=solid}

	\psaxes[dx=2, Dx=0.5, dy=2, Dy=0.5](0.0,0.0)(0.0,0.0)(4.0,4.0)
	\psdot(1.5,1.0)
	
\end{pspicture}
			}
		}
		\subfigure[$(z, G(z)) = (\frac{1}{8}, \frac{1}{4})$]{
			\scalebox{0.7}{
				\begin{pspicture}(-1.5,-0.5)(5.5,4.0)
	\psset{linewidth=0.04}
	\psset{dotsize=0 6}
	
	
	\pspolygon[dimen=middle, linecolor=lightgray, fillstyle=solid, fillcolor=lightgray](0.0,1.5)(0.0,0.5)(0.5,1.0)
	\psline[linecolor=lightgray](0.5,1.0)(1.5,0.0)(4.0,2.5)
	
	\psset{linestyle=dotted}
	
	\psline(0.0,0.0)(2.0,2.0)(4.0,0.0)
	\psline(0.0,2.0)(1.0,1.0)
	\psline(4.0,2.0)(3.0,1.0)
	\psline(0.0,4.0)(1.0,1.0)
	\psline(4.0,4.0)(3.0,1.0)
	\psline(0.0,4.0)(2.0,2.0)
	\psline(4.0,4.0)(2.0,2.0)
	
	\psset{linestyle=solid}
	
	\psaxes[dx=2, Dx=0.5, dy=2, Dy=0.5](0.0,0.0)(0.0,0.0)(4.0,4.0)
	\psdot(0.5,1.0)
	
\end{pspicture}
			}
		}
		\subfigure[$(z, G(z)) = (\frac{1}{8}, \frac{1}{2})$]{
			\scalebox{0.7}{
				\begin{pspicture}(-1.5,-0.5)(5.5,4.0)
	\psset{linewidth=0.04}
	\psset{dotsize=0 6}
	
	
	\pspolygon[dimen=middle, linecolor=lightgray, fillstyle=solid, fillcolor=lightgray](0.0,2.5)(0.0,1.5)(0.5,2.0)
	\psline[linecolor=lightgray](0.5,2.0)(1.25,2.75)(4.0,0.0)
	
	\psset{linestyle=dotted}
	
	\psline(0.0,0.0)(2.0,2.0)(4.0,0.0)
	\psline(0.0,2.0)(1.0,1.0)
	\psline(4.0,2.0)(3.0,1.0)
	\psline(0.0,4.0)(1.0,1.0)
	\psline(4.0,4.0)(3.0,1.0)
	\psline(0.0,4.0)(2.0,2.0)
	\psline(4.0,4.0)(2.0,2.0)
	
	\psset{linestyle=solid}

	\psaxes[dx=2, Dx=0.5, dy=2, Dy=0.5](0.0,0.0)(0.0,0.0)(4.0,4.0)
	\psdot(0.5,2.0)
	
\end{pspicture}
			}
		}
		\subfigure[$(z, G(z)) = (\frac{1}{4}, \frac{1}{2})$]{
			\scalebox{0.7}{
				\begin{pspicture}(-1.5,-0.5)(5.5,4.0)
	\psset{linewidth=0.04}
	\psset{dotsize=0 6}
	
	
	\pspolygon[dimen=middle, linecolor=lightgray, fillstyle=solid, fillcolor=lightgray](1.0,2.0)(1.5,2.5)(3.75,0.25)(4.0,0.5)(4.0,0.0)(3.0,0.0)
	\psline[linecolor=lightgray](0.0,3.0)(3.0,0.0)
	
	\psset{linestyle=dotted}
	
	\psline(0.0,0.0)(2.0,2.0)(4.0,0.0)
	\psline(0.0,2.0)(1.0,1.0)
	\psline(4.0,2.0)(3.0,1.0)
	\psline(0.0,4.0)(1.0,1.0)
	\psline(4.0,4.0)(3.0,1.0)
	\psline(0.0,4.0)(2.0,2.0)
	\psline(4.0,4.0)(2.0,2.0)
	
	\psset{linestyle=solid}

	\psaxes[dx=2, Dx=0.5, dy=2, Dy=0.5](0.0,0.0)(0.0,0.0)(4.0,4.0)
	\psdot(1.0,2.0)
	
\end{pspicture}
			}
		}
	\caption{In grey, those locations for the second data point in the one-dimensional Example \ref{eg:Phat_1dim_example} that are non-binding with respect to the first point (the black dot);  \emph{cf.}\ (a)--(d) of Figure \ref{fig:discty}.}
	\label{fig:second}
\end{figure}

A sufficient (but not necessary) condition for the extreme value of an optimization problem to be unchanged upon the introduction of a new constraint is that the extremizer of the original problem is feasible with respect to the new constraint.  This, a sufficient condition for a data point to be non-binding is provided by the following result:

\begin{prop}[Non-binding data points]
	\label{prop:non-binding}
	Let $\mathcal{O} \subseteq \mathcal{X}$, $z_{0} \in \mathcal{X}$, $L$ and $G|_{\mathcal{O} \cup \{ z_{0} \}}$ be given.
	\begin{enumerate}
		\item Let $(\bar{x}, \bar{y}, \bar{x}', \bar{y}')$ be a maximizer for \eqref{eq:opt_diam_1} with observations $\mathcal{O}$.  If
		\begin{equation}
			\label{eq:non-binding_D}
			| \bar{y} - G(z_{0}) | \leq d_{L}(\bar{x}, z_{0}) \text{ and } | \bar{y}' - G(z_{0}) | \leq d_{L}(\bar{x}', z_{0}),
		\end{equation}
		then $z_{0}$ is non-binding and $\widehat{D}_{k}[\mathcal{X}, G|_{\mathcal{O} \cup \{ z_{0} \}}, L] = \widehat{D}_{k}[\mathcal{X}, G|_{\mathcal{O}}, d_{L}]$.
		\item Let $(\bar{x}_{0}, \bar{x}_{1}, \bar{y}, \bar{p})$ be a maximizer for \eqref{eq:opt_prob_reduced} with observations $\mathcal{O}$.  If
		\begin{equation}
			\label{eq:non-binding_P}
			| \bar{y}_{\eps} - G(z_{0}) | \leq d_{L}(\bar{x}_{\eps}, z_{0}) \text{ for all $\eps \in \{ 0, 1 \}^{K}$,}
		\end{equation}
		then $z_{0}$ is non-binding and $\widehat{P}[\mathcal{X}, G|_{\mathcal{O} \cup \{ z_{0} \}}, L, m] = \widehat{P}[\mathcal{X}, G|_{\mathcal{O}}, L, m]$.
	\end{enumerate}
\end{prop}

\begin{proof}
	Since $\mathcal{O} \subseteq \mathcal{O} \cup \{ z_{0} \}$, every $(x, y, x', y')$ that is feasible for \eqref{eq:opt_diam_1} with observations $\mathcal{O} \cup \{ z_{0} \}$ is also feasible for \eqref{eq:opt_diam_1} with observations $\mathcal{O}$.  Hence
	\[
		\widehat{D}_{k}[\mathcal{X}, G|_{\mathcal{O}}, d_{L}] \geq \widehat{D}_{k}[\mathcal{X}, G|_{\mathcal{O} \cup \{ z_{0} \}}, L].
	\]
	Now let $(\bar{x}, \bar{y}, \bar{x}', \bar{y}')$ be a maximizer for \eqref{eq:opt_diam_1} with observations $\mathcal{O}$ and suppose that \eqref{eq:non-binding_D} holds;  then $(\bar{x}, \bar{y}, \bar{x}', \bar{y}')$ satisfies the criteria to be a feasible point for \eqref{eq:opt_diam_1} with observations $\mathcal{O} \cup \{ z_{0} \}$, and has the same objective function value $| \bar{y} - \bar{y}' |$.  Hence,
	\[
		\widehat{D}_{k}[\mathcal{X}, G|_{\mathcal{O}}, d_{L}] \leq \widehat{D}_{k}[\mathcal{X}, G|_{\mathcal{O} \cup \{ z_{0} \}}, L],
	\]
	and the claim for $\widehat{D}_{k}$ follows.  The proof of the claim for $\widehat{P}$ is analogous.
\end{proof}

Note well that the converse of Proposition \ref{prop:non-binding} is false in general:  the introduction of a new data point may render some of the previous (non-unique) maximizers infeasible but still fail to change the maximum value of the problem.

Nevertheless, the simple algebraic conditions of Proposition \ref{prop:non-binding} suggest a practical method for calculating $\widehat{D}_{k}$ or $\widehat{P}$ if the data set $\mathcal{O} = \{ z_{1}, \dots, z_{N} \}$ is a large finite set that is believed to contain many redundant points.  The idea is to introduce the data points one at a time and only solve \eqref{eq:opt_diam_1} (for $\widehat{D}_{k}$) or \eqref{eq:opt_prob_reduced} (for $\widehat{P}$) when strictly necessary.  In the following algorithm, $\mathcal{O}_{i} \subseteq \mathcal{O}$ will denote the data points (constraints) that are enforced at iteration $i$, while $\widetilde{\mathcal{O}}_{i} \subseteq \mathcal{O}$ will denote those that are potentially binding and will be checked for feasibility at iteration $i$.  Note well that, in general, $\mathcal{O}_{i} \cup \widetilde{\mathcal{O}}_{i} \subsetneq \mathcal{O}$.

\begin{alg}
	\label{alg:legacy_ouq_simplex}
	Initialize with $\mathcal{O}_{0} = \varnothing$ and $\widetilde{\mathcal{O}}_{0} = \mathcal{O}$.  Then, for $i = 1, 2, \dots$,
	\begin{enumerate}
		\item For each $z \in \widetilde{\mathcal{O}}_{i - 1}$, calculate $\widehat{D}_{k}[\mathcal{X}, G|_{\mathcal{O}_{i - 1} \cup \{ z \}}, L]$.
		\item Let $\mathcal{M} \subseteq \widetilde{\mathcal{O}}_{i - 1}$ be the set of maximizers of $z \mapsto \widehat{D}_{k}[\mathcal{X}, G|_{\mathcal{O}_{i - 1} \cup \{ z \}}, L]$ among $z \in \widetilde{\mathcal{O}}_{i - 1}$.
		\item Set $\mathcal{O}_{i} := \mathcal{O}_{i - 1} \cup \mathcal{M}$ and calculate $\widehat{D}_{k}[\mathcal{X}, G|_{\mathcal{O}_{i}}, L]$.
		\item Let $\widetilde{\mathcal{O}}_{i}$ consist of those $z \in \mathcal{O} \setminus \mathcal{O}_{i}$ such that the extremizer for $\widehat{D}_{k}[\mathcal{X}, G|_{\mathcal{O}_{i}}, L]$ is infeasible with respect to $(z, G(z))$ (\emph{i.e.}\ fails \eqref{eq:non-binding_D}), and hence is possibly binding.
		\item Terminate if $\widetilde{\mathcal{O}}_{i} = \varnothing$.
	\end{enumerate}
\end{alg}

The algorithm for $\widehat{P}$ is analogous, with \eqref{eq:non-binding_P} in place of \eqref{eq:non-binding_D}.

In the numerical examples that have been considered so far, it has been observed that relatively few elements of $\mathcal{O}$ determine $\widehat{D}_{k}$ or $\widehat{P}$, even though, in principle, every element of $\mathcal{O}$ could supply a binding constraint.  This situation is somewhat analogous to the simplex algorithm in linear programming:  in the theoretical worst case, the simplex method can take exponential time \cite{KleeMinty:1972}, but it ``usually'' requires polynomial time in practice.  We will reserve detailed numerical analysis of this algorithm for a future work.

\section{Further Remarks}
\label{sec:further_remarks}

\subsection{Feasible Lipschitz Constants}

Given $\mathcal{O} \subseteq \mathcal{X}$ and the associated observations $G|_{\mathcal{O}}$, let $\mathrm{Lip}(G|_{\mathcal{O}})$ denote the set of Lipschitz constants for $G$ that are consistent with the observations $G|_{\mathcal{O}}$, \emph{i.e.}
\begin{equation}
	\label{eq:feasible_L}
	\mathrm{Lip}(G|_{\mathcal{O}}) := \left\{ L \in \R^{K} \smid \begin{matrix} \text{for all $z, z' \in \mathcal{O}$,} \\ | G(z) - G(z') | \leq d_{L}(z, z') \end{matrix} \right\}.
\end{equation}
It is easy to check that, for any given $\mathcal{O} \subseteq \mathcal{X}$ and $G|_{\mathcal{O}}$, $\mathrm{Lip}(G|_{\mathcal{O}})$ is a convex subset of $\R^{K}$.  This remains the case if additional inequality constraints on the $L_{k}$ are supplied:  \emph{e.g.}\ if it is required that $\ell_{k}^{-} \leq L_{k} \leq \ell_{k}^{+}$, then
\[
	\mathrm{Lip}'(G|_{\mathcal{O}}) := \left\{ L \in \mathrm{Lip}(G|_{\mathcal{O}}) \smid \ell_{k}^{-} \leq L_{k} \leq \ell_{k}^{+} \text{ for each } k \in \{ 1, \dots, K \} \right\}
\]
is a convex set.

It is not immediately clear what one should regard as the ``smallest'' element of $\mathrm{Lip}(G|_{\mathcal{O}})$.   However, recall that Theorem \ref{thm:Dhat_error} shows that the gap size $\Gamma$ of the data set with respect to $d_{L}$ controls the error $\widehat{D}_{k} - \mathcal{D}_{k}[G]$:
\[
	0 \leq \widehat{D}_{k}[\mathcal{X}, G|_{\mathcal{O}}, d_{L}] - \mathcal{D}_{k}[G] \leq 4 \Gamma(\mathcal{X}, \mathcal{O}, d_{L}).
\]
Therefore, it makes sense to search among the feasible Lipschitz constants $L \in \mathrm{Lip}(G|_{\mathcal{O}})$ for one $L^{\ast}$ that minimizes the gap size.  Unfortunately, this is not a \emph{convex minimization problem} in the sense of \cite[\S4.2]{BoydVandenberghe:2004}, since $\Gamma(\mathcal{X}, \mathcal{O}, d_{L})$ is not a convex function of $L$:  for each $x \in \mathcal{X}$, $d_{L}(x, \mathcal{O})$ is a concave function of $L$, and a supremum of a family of concave functions can be badly behaved.  $\widehat{D}_{k}[\mathcal{X}, G|_{\mathcal{O}}, d_{L^{\ast}}]$ is then the upper bound on $\mathcal{D}_{k}[G]$ that has the tightest error estimate that can be justified by the data $G|_{\mathcal{O}}$ alone;  of course, further data might invalidate this scenario.

\subsection{Sensitivity and Robustness Analysis}

In some applications, there may be doubt about the correct values for the Lipschitz constants $L_{1}, \dots, L_{K}$.  Such doubt necessarily propagates to doubt about the validity of the bounds $\widehat{D}_{k}$ and $\widehat{P}$:  however, it does not do so in an entirely uncontrolled fashion.  It is possible to perform a (local or global) sensitivity/robustness analysis of $\widehat{D}_{k}$ and $\widehat{P}$ with respect to $L_{1}, \dots, L_{K}$ and thereby determine which Lipschitz constants strongly control the values of $\widehat{D}_{k}$ and $\widehat{P}$;  the key Lipschitz constants can be identified for further, more detailed, research;  the less important ones can be (relatively) safely accepted as they stand.  

Notably, as in the optimal concentration-of-measure inequalities of McDiarmid and Hoeffding type \cite[\S4]{OwhadiScovelSullivanMcKernsOrtiz:2010}, some $L_{k}$ may turn out to have \emph{zero} influence on $\widehat{D}_{k}$ and $\widehat{P}$.  Indeed, by rescaling arguments, it is easy to see that just as $\widehat{D}_{k}$ and $\widehat{P}$ may be discontinuous as functions of the observed data $G|_{\mathcal{O}}$ (as in Example \ref{eg:Phat_1dim_example}), $\widehat{D}_{k}$ and $\widehat{P}$ may be discontinuous as functions of $L$.

\section{Numerical Examples}
\label{sec:numerics}

This section covers the numerical calculation of $\widehat{P}$ in two example cases.  The first case (Subsection \ref{subsec:numerics_1d}) is a validation exercise, in which the closed-form results of Example \ref{eg:Phat_1dim_example} are replicated numerically.  The second case (Subsection \ref{subsec:numerics_PSAAP}) is a more involved calculation, in which the response function is a function of three variables and the data set comes from an archive of impact engineering experiments.

\subsection{Overview of the Numerical Method}
\label{subsec:numerics_overview}

A description of the OUQ algorithm, as implemented in the \emph{mystic} framework \cite{McKernsHungAivazis:2009}, can be found in \cite{McKernsOwhadiScovelSullivanOrtiz:2010, McKernsStrandSullivanFangAivazis:2011}.  In those earlier implementations of OUQ, it was the case that the response function was known/modelled exactly, and so it was only necessary to numerically represent the unknown probability measure $\mu$.  To implement the ``Legacy OUQ'' method of this paper, it was necessary to extend the existing OUQ algorithm in the following ways:
\begin{itemize}
	\item \emph{Mystic}'s \verb!product_measure! class, which provides a numerical representation of a probability measure $\mu$ of the form \eqref{eq:prod_representation}/\eqref{eq:sum_representation}, was extended to associate to each of the support points of a product measure $\mu$ a scalar value, thereby providing a numerical representation of a pair $(g, \mu) \in \mathcal{A}_{\Delta}$ as in \eqref{eq:opt_prob_reduced_feasible}.  Such an object will be referred to as a \verb!scenario! and denoted \verb!X!;  typically, \verb!X! is stored in the ``compressed'' form of $(x_{0}, x_{1}, p, y)$ as used in \eqref{eq:opt_prob_reduced} and elsewhere, but is sometimes converted into other representations.
	\item A \verb!dataset! class, which numerically represents the observed data $G|_{\mathcal{O}}$ and the cone structure that comes from the Lipschitz constants, was added.  As alluded to in the previous bullet point, a \verb!scenario! object \verb!X! can be regarded as a \verb!dataset! object by ``forgetting'' the probabilistic structure and remembering only the points in input parameter space and their associated output values.  Below, the legacy data set $G|_{\mathcal{O}}$ will be denoted \verb!data!.
	\item Methods were added to both of these classes to allow for efficient calculation of $d_{L}$ distances (and hence whether or not a given \verb!scenario! object \verb!X! is $d_{L}$-short with respect to itself and \verb!data!) and integrals with respect to $\mu$ as in \eqref{eq:sum_representation_r}.
\end{itemize}

The overall structure of the optimization calculations is that of an outer and an inner optimization loop.  The outer loop generates the next population of candidate \verb!scenario! objects \verb!X! to which the objective function \verb!F! (the probability-of-failure functional) will be applied.  The inner loop applies the constraints (bounds, mean, and shortness) to those generated candidates \verb!X! so that \verb!F! is only ever evaluated on \verb!scenario! objects \verb!X'=C(X)! that satisfy the constraints imposed by \verb!C!.

The outer optimization loop, as described in \cite{McKernsOwhadiScovelSullivanOrtiz:2010}, is used with the ``expanded solver interface'' described in \cite{McKernsStrandSullivanFangAivazis:2011}.  A differential evolution solver \cite{PriceStornLampinen:2005, StornPrice:1997} was used with termination condition \verb!ChangeOverGenerations!, population size \verb!npop!~$= 32$, \verb!ngen!~$= 100$, and \verb!tol!~$= 10^{-6}$;  that is, the calculations used populations of 32 candidates and terminated when the best objective function value had shown no improvement greater than $10^{-6}$ for 100 consecutive iterations of the outer loop.  The objective function value \verb!F(X)!, when \verb!X! represents $(g, \mu)$, is the probability of failure for $g$ under $\mu$ as defined in \eqref{eq:sum_representation_r} with $r(x) := \one[g(x) \leq \theta]$.  The optimizer generates values for the weights and positions of the measure points in each coordinate direction.  For Legacy OUQ, the optimizer must also generate scalar values $y = g(x)$ for each point $x$ in the support of the product measure $\mu$.

In \emph{mystic}, constraints are solved explicitly through algebraic or numerical means.  A \emph{constraints solver} \verb!C! is built to impose the set of constraints on the candidate \verb!scenario! generated by the outer loop optimizer at each iteration.  Constraints solvers are functions that map any (not necessary feasible) \verb!scenario! object \verb!X! to a \verb!scenario! object \verb!X'=C(X)! that satisfies all of the required constraints.  Thus, only valid solutions to the constraints equations are seen by the objective function \verb!F!.  Effectively, the value of the objective function value evaluated by the outer loop optimizer at each step is \verb!F(C(X))!.  In contrast, standard optimizers use penalty functions \verb!P! (and often dynamic multipliers \verb!k!) so that the objective function \verb!F! as evaluated by the optimizer is in fact \verb!F(X)+k*P(X)!;  this approach corrupts the structure of the problem by severing an explicit connection to the constraints.

The constraints function used in the Legacy OUQ algorithm first builds the \verb!scenario! object \verb!X! from the optimizer-generated inputs to the objective function.  A first constraints solver \verb!C'! is then applied:  this ensures that the weights of each of the underlying discrete measures sum to \verb!1.0!.  A second constraints solver \verb!C''! is then applied, which imposes the mean constraint $\E_{\mu}[g] \geq m$;  this is done through \emph{mystic}'s \verb!impose_mean! function, which, in our example, shifts the coordinates of \verb!X! so that \verb!X! has the desired mean.  At this point, the candidate \verb!scenario! objects \verb!X! generated by the optimizer have passed through the constraints solvers \verb!C'! and \verb!C''!, and only provide the objective function \verb!F! with valid solutions \verb!X'=C*(X)=C''(C'(X))! of the given bounds and mean constraints.  If the resulting candidate \verb!scenario! object \verb!X'! is not $d_{L}$-short with respect to itself and to the legacy data \verb!data!, \emph{i.e.}\ the inequality
\[
	| g(x) - g(x') | \leq d_{L}(x, x')
\]
fails for some $x$ in the support of $\mu$ and some $x'$ either in the support of $\mu$ or in $\mathcal{O}$, then \emph{mystic}'s \verb!set_feasible! function is used in a third constraints solver \verb!C! to impose the desired shortness on the \verb!scenario! object \verb!X'!.  Unlike for \verb!C'! and \verb!C''!, the constraints in \verb!C! can not be imposed algebraically.  Instead, the application of \verb!C! is an inner optimization loop. 

The details of how \emph{mystic} checks for shortness and how feasibility
is imposed on a \verb!scenario! object are worth a little further discussion.  

The check for shortness of a scenario \verb!X! with respect to the legacy data \verb!data! is done by first converting \verb!X! into a \verb!dataset! object with the \verb!load! method, and then applying the \verb!is_short! function, which calculates the a 2-dimensional array \verb!dist! with elements $|y - y'| - d_L(x,x')$ for each combination of $x, x'$ from the two collections of support points (here, the legacy data set \verb!data! and the scenario \verb!X! regarded as a data set).  The result is a matrix corresponding to the distances required for shortness, where all distances less than a given tolerance \verb!short_tol! are treated as acceptably close to zero;  if all entries of the matrix \verb!dist! are at most \verb!short_tol!, then, modulo that tolerance, \verb!X! is $d_{L}$-short with respect to \verb!data!;  otherwise, the positivity of the matrix \verb!dist! provides a numerical measure of the failure of shortness.  Shortness of \verb!X! with respect to itself is calculated similarly.

Shortness is imposed through an inner optimization loop that solves for a candidate \verb!scenario! object \verb!X'! for which \verb!dist<=short_tol!. Similarly to the outer optimization loop, this inner optimization loop uses a differential evolution solver --- however, the termination condition used in the inner loop is \verb!VTR! \cite{McKernsHungAivazis:2009}, and solver parameters are set to \verb!npop!~$= 40$ and \verb!tol!~$= 10^{-9}$.  The constraints solver \verb!C*! described above is reused by the inner optimization loop to ensure that the constraints on the weights and mean are also respected by \verb!C!.  For shortness, the objective function for the inner loop is the sum over all elements of the matrix \verb!max(0.0, dist-short_tol)!.  When the inner loop terminates, a candidate \verb!scenario! object \verb!X'=C(X)! is produced that satisfies all constraints imposed by the solver \verb!C! (and thus also \verb!C*!).

The solution produced by the outer optimization loop is a \verb!scenario! object \verb!C(X)! that both satisfies all of the above constraints and maximizes the probability of failure \verb!F(C(X))!.

\subsection{One Data Point in One Dimension}
\label{subsec:numerics_1d}

As a first exercise in applying the protocol, we  numerically replicate the exact values for $\widehat{P}$ in Example \ref{eg:Phat_1dim_example}.  Numerical convergence plots are given in Figure \ref{fig:Phat_1dim_example_cvgce}.  In this subsection and the next, $\widehat{P}_{n}$ denotes the optimizer's best approximation to $\widehat{P}$ after $n$ outer loop iterations.

It may be useful to note that the dimensionality of the problem can be slightly reduced, and more accurate results obtained more quickly, if instead of searching over
\[
	(x_{0}, x_{1}, y_{0}, y_{1}, p) \in [0, 1]^{2} \times \R^{2} \times [0, 1] \text{,}
\]
one instead forces $(x_{1}, y_{1})$ to be a failure, and therefore searches over
\[
	(x_{0}, x_{1}, y_{0}, y_{1}, p) \in [0, 1]^{2} \times \R \times \{ 0 \} \times [0, 1] \text{.}
\]
The same value for $\widehat{P}$ is attained using either approach;  if $y = 0$ is not a feasible value for any $x \in [0, 1]$, then the optimizer detects this fact and reports that the feasible set is empty, from which we infer that the maximum probability of failure is zero.

\begin{figure}
	\scalebox{0.75}{
		\input{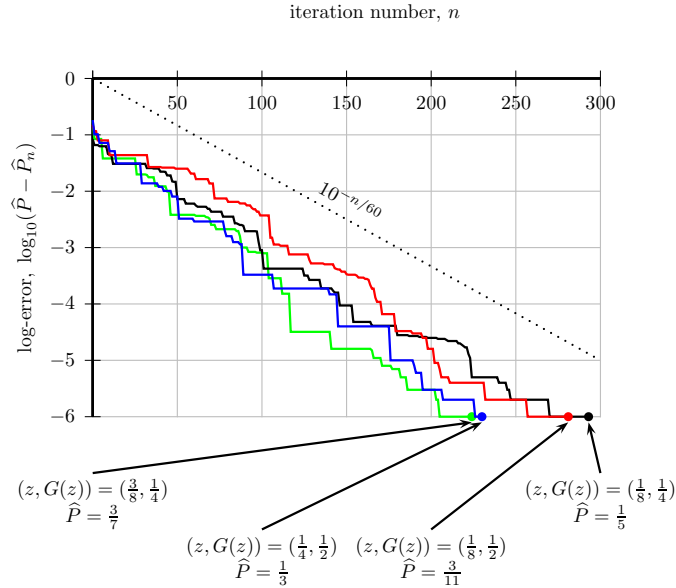}
	}
	\caption{Log-linear plot illustrating typical numerical convergence of the approximate maxima $\widehat{P}_{n}$ as a function of the number $n$ of outer loop iterations in the numerical implementation of Example \ref{eg:Phat_1dim_example}.  Note the approximate convergence rate of $| \widehat{P}_{n} - \widehat{P} | \approx 10^{- ( 1 + n/ 60 )}$.  After the last iteration shown in each plot, $| \widehat{P}_{n} - \widehat{P} | \leq 10^{-6}$, \emph{i.e.}\ the two are equal up to the convergence tolerance.}
	\label{fig:Phat_1dim_example_cvgce}
\end{figure}

\subsection{Three-Dimensional Example}
\label{subsec:numerics_PSAAP}

This subsection reports the results of implementing the above method for obtaining optimal bounds on probabilities using a data set generated by physical experiments.  These experiments were performed at the California Institute of Technology's Small Particle Hypervelocity Impact Range (SPHIR) facility.  A brief description of the experimental setup is given in the next two paragraphs;  the essential mathematical point is that Table \ref{tbl:psaap-data} forms the legacy data for a function $G$ of three real-valued inputs with smoothness given by \eqref{eq:PSAAP_Lip}.

In these experiments, a solid steel ball of diameter $0.07\,$inches is fired at an aluminium plate of thickness $h$.  The projectile impacts the plate at an angle $\alpha$ away from the plate normal (referred to as the \emph{obliquity} of the impact), and at a speed $v$.  This impact event may result in the plate being perforated\footnote{
To be precise, \emph{perforation} (also known as \emph{complete penetration}) means that the impact event has caused a hole in the plate that passes fully from one side of the plate to the other;  a topologist would say that the plate has changed topology from genus $0$ to genus $\geq 1$.  The opposite situation, in which the plate is merely ``dented'' by the projectile, is referred to as a \emph{penetration} or \emph{partial penetration}.  The shorthand terms ``a complete'' and ``a partial'' are in common use.} by the projectile.  

The impact event is very complicated, with many physical processes happening at very high rates;  the experimental diagnostics and the numerical modelling of the entire event are beyond the scope of this paper, and further details can be found in \cite{ALLMMOORSS:2011, KLLMOORSS:2011}.  To simplify matters and focus on the relevant mathematics, this example selects a single, scalar, ``post mortem'' quantity of interest:  after the impact event, the cross-sectional area $G(h, \alpha, v)$ (in $\mathrm{mm}^{2}$) of the perforation in the plate is measured using an optical scanner and recorded, with the obvious convention that failure to perforate means that $G(h, \alpha, v) = 0$.  

\begin{table}[t]
	\caption{Hypervelocity impact legacy data.  Note that this data set corresponds to a multi-valued function: see shots A62 and A77.}
	\label{tbl:psaap-data}
	\begin{tabular}{|c|rrr|r|}
	\hline
	ID & \multicolumn{1}{c}{plate thickness} & \multicolumn{1}{c}{impact obliquity} & \multicolumn{1}{c|}{impact speed} & \multicolumn{1}{c|}{perforation area} \\
	~ & \multicolumn{1}{c}{$h / \mathrm{in}$} & \multicolumn{1}{c}{$\alpha / \mathrm{deg}$} & \multicolumn{1}{c|}{$v / \mathrm{m} \cdot \mathrm{s}^{-1}$} & \multicolumn{1}{c|}{$G(h, \alpha, v) / \mathrm{mm}^{2}$} \\
	\hline \hline
	A48 & 0.062 & 0.0 & 2288.0 & 7.73 \\
	A49 & 0.125 & 30.0 & 2840.0 & 13.38 \\
	A50 & 0.125 & 0.0 & 2556.0 & 11.83\\
	A51 & 0.062 & 30.0 & 2329.0 & 6.31 \\
	A52 & 0.062 & 0.0 & 2363.0 & 7.78 \\
	A53 & 0.125 & 0.0 & 2326.0 & 9.26 \\
	A54 & 0.125 & 30.0 & 3235.0 & 15.98 \\
	A55 & 0.062 & 0.0 & 2686.0 & 9.86 \\
	A56 & 0.062 & 30.0 & 2728.0 & 11.35 \\
	A57 & 0.062 & 30.0 & 2627.0 & 12.09 \\
	A58 & 0.125 & 30.0 & 2531.0 & 11.24 \\
	A60 & 0.125 & 0.0 & 2363.0 & 9.93 \\
	A61 & 0.062 & 0.0 & 2707.0 & 9.96 \\
	A62 & 0.062 & 30.0 & 2756.0 & 11.07 \\
	A63 & 0.062 & 0.0 & 2614.0 & 9.02 \\
	A64 & 0.125 & 0.0 & 2439.0 & 10.52 \\
	A65 & 0.062 & 0.0 & 2485.0 & 8.56 \\
	A66 & 0.125 & 0.0 & 2607.0 & 12.46 \\
	A67 & 0.125 & 30.0 & 3036.0 & 15.36 \\
	A68 & 0.125 & 30.0 & 2325.0 & 8.15 \\
	A69 & 0.062 & 30.0 & 2702.0 & 10.81 \\
	A70 & 0.062 & 30.0 & 2473.0 & 9.52 \\
	A71 & 0.121 & 30.0 & 2520.0 & 9.47 \\
	A72 & 0.121 & 0.0 & 2439.0 & 10.19 \\
	A73 & 0.121 & 30.0 & 2366.0 & 9.42 \\
	A74 & 0.121 & 30.0 & 2402.0 & 8.68 \\
	A75 & 0.062 & 30.0 & 2413.0 & 9.19 \\
	A77 & 0.062 & 30.0 & 2756.0 & 11.32 \\
	A78 & 0.121 & 30.0 & 2432.0 & 10.00 \\
	A79 & 0.062 & 30.0 & 2393.0 & 9.29 \\
	A80 & 0.121 & 30.0 & 2479.0 & 9.53 \\
	A81 & 0.060 & 30.0 & 2356.0 & 8.27 \\
	\hline
\end{tabular}

\end{table}

The results of a series of such impact tests are given in Table \ref{tbl:psaap-data}, which forms the legacy data set $G|_{\mathcal{O}}$ for this example.  The protocol described above is now applied over the parameter space
\[
	(h, \alpha, v) \in \mathcal{X} := [0.062, 0.125] \, \mathrm{in} \times [0, 30] \, \mathrm{deg} \times [2300, 3200] \, \mathrm{m} \cdot \mathrm{s}^{-1}.
\]
The data are, in fact, multi-valued (two distinct perforation areas were observed for the same input triplet $(h, \alpha, v)$).  Therefore, the response function is not Lipschitz continuous, and so we apply a natural generalization of the above protocol using the following ``Lipschitz with tolerance'' constraint:
\begin{equation}
	\label{eq:PSAAP_Lip}
	| G(h, \alpha, v) - G(h', \alpha', v') | \leq d_{L}((h, \alpha, v), (h', \alpha', v')) + T \text{,}
\end{equation}
where
\begin{align*}
	L & := (L_{h}, L_{\alpha}, L_{v}) \text{,} &
	T & := 1.0 \, \mathrm{mm}^{2} \text{,}
\end{align*}
\begin{align*}
	L_{h} & := 175.0 \, \mathrm{mm}^{2} / \mathrm{in}, &
	L_{\alpha} & := 0.075 \, \mathrm{mm}^{2} / \mathrm{deg}, &
	L_{v} & := 0.1 \, \mathrm{mm}^{2} / (\mathrm{m} \cdot \mathrm{s}^{-1}) \text{.}
\end{align*}
Condition \eqref{eq:PSAAP_Lip} is satisfied by the observed data in Table \ref{tbl:psaap-data}, and we assume that it remains valid for the system in operation.  We also assume that the system in operation will be exposed to random $(h, \alpha, v)$ taking values in $\mathcal{X}$, with independent components, and such that $\E[G(h, \alpha, v)] \geq 11.0 \, \text{mm}^{2}$.

In this example, the ``failure'' event is that the perforation area $G(h, \alpha, v)$ falls below some threshold area $\theta$.  Figure \ref{fig:num_results_thetas} shows the computed least upper bound on $\P[G(h, \alpha, v) \leq \theta]$ for $\theta \in \{ 0, 1, \dots, 12 \} \, \mathrm{mm}^{2}$.  As expected, the least upper bound on $\P[G(h, \alpha, v) \leq \theta]$ is indeed $1$ when $\theta \geq m$ and decreases as $m - \theta$ increases.

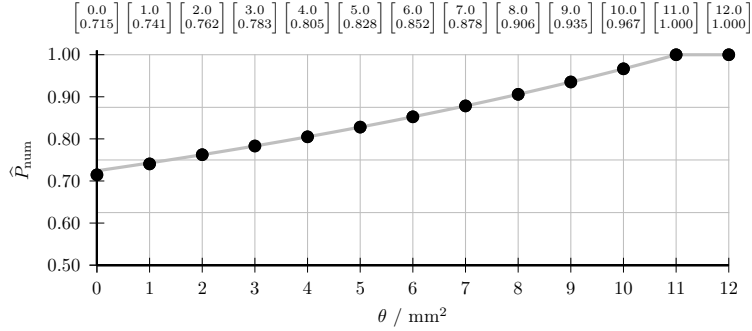
\begin{figure}
	\scalebox{0.7}{
		\begin{pspicture}(-2.0,3.0)(14.0,9.5)
	\psset{linewidth=0.04}
	\psset{dotsize=0 6}
	
	\psgrid[subgridwidth=0pt, subgriddots=0, subgridcolor=white, gridlabels=0, gridcolor=lightgray, gridwidth=0.02cm, xunit=1cm, yunit=1cm](0,4)(0,4)(12,8)
	
	\psset{yunit=8.0}
	\psaxes[dy=0.10, Dy=0.10, Oy=0.50](0.0,0.5)(0.0,0.5)(12.0,1.01)
	\rput[b]{90}(-1.25,0.75){$\widehat{P}_{\text{num}}$}
	\rput[t](6.0,0.40){$\theta$ / $\text{mm}^{2}$}
	
	\psplot[linecolor=lightgray, linewidth=0.06cm]{0.0}{11.0}{39.895 11.00 sub 39.895 x sub div}
	\psplot[linecolor=lightgray, linewidth=0.06cm]{11.0}{12.0}{1.0}

	\rput[b](0.0,1.050){\scriptsize $\begin{bmatrix} 0.0 \\ 0.715 \end{bmatrix}$}
	\psdot[dotstyle=Bo, fillcolor=black, linecolor=black](0.0,0.714693)

	\rput[b](1.0,1.050){\scriptsize $\begin{bmatrix} 1.0 \\ 0.741 \end{bmatrix}$}
	\psdot[dotstyle=Bo, fillcolor=black, linecolor=black](1.0,0.740630)
	
	\rput[b](2.0,1.050){\scriptsize $\begin{bmatrix} 2.0 \\ 0.762 \end{bmatrix}$}
	\psdot[dotstyle=Bo, fillcolor=black, linecolor=black](2.0,0.762502)
	
	\rput[b](3.0,1.050){\scriptsize $\begin{bmatrix} 3.0 \\ 0.783 \end{bmatrix}$}
	\psdot[dotstyle=Bo, fillcolor=black, linecolor=black](3.0,0.783168)

	\rput[b](4.0,1.050){\scriptsize $\begin{bmatrix} 4.0 \\ 0.805 \end{bmatrix}$}
	\psdot[dotstyle=Bo, fillcolor=black, linecolor=black](4.0,0.804987)
	
	\rput[b](5.0,1.050){\scriptsize $\begin{bmatrix} 5.0 \\ 0.828 \end{bmatrix}$}
	\psdot[dotstyle=Bo, fillcolor=black, linecolor=black](5.0,0.828055)

	\rput[b](6.0,1.050){\scriptsize $\begin{bmatrix} 6.0 \\ 0.852 \end{bmatrix}$}
	\psdot[dotstyle=Bo, fillcolor=black, linecolor=black](6.0,0.852486)

	\rput[b](7.0,1.050){\scriptsize $\begin{bmatrix} 7.0 \\ 0.878 \end{bmatrix}$}
	\psdot[dotstyle=Bo, fillcolor=black, linecolor=black](7.0,0.878401)

	\rput[b](8.0,1.050){\scriptsize $\begin{bmatrix} 8.0 \\ 0.906 \end{bmatrix}$}
	\psdot[dotstyle=Bo, fillcolor=black, linecolor=black](8.0,0.905941)

	\rput[b](9.0,1.050){\scriptsize $\begin{bmatrix} 9.0 \\ 0.935 \end{bmatrix}$}
	\psdot[dotstyle=Bo, fillcolor=black, linecolor=black](9.0,0.935265)

	\rput[b](10.0,1.050){\scriptsize $\begin{bmatrix} 10.0 \\ 0.967 \end{bmatrix}$}
	\psdot[dotstyle=Bo, fillcolor=black, linecolor=black](10.0,0.966550)

	\rput[b](11.0,1.050){\scriptsize $\begin{bmatrix} 11.0 \\ 1.000 \end{bmatrix}$}
	\psdot[dotstyle=Bo, fillcolor=black, linecolor=black](11.0,1.000000)

	\rput[b](12.0,1.050){\scriptsize $\begin{bmatrix} 12.0 \\ 1.000 \end{bmatrix}$}
	\psdot[dotstyle=Bo, fillcolor=black, linecolor=black](12.0,1.000000)

\end{pspicture}
	}
	\caption{Numerical results for the least upper bound on $\P[G(h, \alpha, v) \leq \theta]$ for various $\theta$.  Note the close agreement with the Markov bound \eqref{eq:Markov_bound} (grey line):  for $\theta \geq 2.0 \, \mathrm{mm}^{2}$, the difference is less than the change-over-generations criterion of $10^{-6}$.  Convergence plots are given in Figure \ref{fig:num_results_cvgce_compare}.}
	\label{fig:num_results_thetas}
\end{figure}

\begin{rmk}[Markov bound and non-binding data]
	One interesting feature of Figure \ref{fig:num_results_thetas} is that the numerical results demonstrate very close agreement with the Markov bound
	\begin{equation}
		\label{eq:Markov_bound}
		\P[G(h, \alpha, v) \leq \theta] \leq \frac{M - m}{M - \theta} \text{,}
	\end{equation}
	where
	\begin{equation}
		\label{eq:Markov_maximum}
		M := \sup_{(h, \alpha, v) \in \mathcal{X}} \inf_{z \in \mathcal{O}} \big( G(z) + d_{L}(z, (h, \alpha, v)) + T \big) \approx 39.895 \, \mathrm{mm}^{2}
	\end{equation}
	with maximizer at
	\begin{equation}
		\label{eq:Markov_maximizer}
		(h_{M}, \alpha_{M}, v_{M}) \approx (0.062 \, \mathrm{in}, 0.0 \, \mathrm{deg}, 3138.6 \, \mathrm{m} \cdot \mathrm{s}^{-1})
	\end{equation}
	is the largest perforation area that can be realised anywhere in $\mathcal{X}$ subject to the data and the Lipschitz constraints.  (We note in passing that efficient algorithms for finding extrema of Lipschitz functions are an area of independent interest:  see \emph{e.g.}\ \cite{JonesPerttunenStuckman:1993}.)  Indeed, for $\theta \geq 2.0 \, \text{mm}^{2}$, the difference between the computed $\widehat{P}$ and Markov's bound is dominated by the numerical convergence criterion (less than \verb!tol!~$=10^{-6}$ change over \verb!ngen!~$=10^{2}$ consecutive generations).
	
	This observation shows that most of the data set (\emph{i.e.}\ those data points that do not determine $M$) consists of non-binding data points;  indeed, only the constraints corresponding to data points A54 and A67 in Table \ref{tbl:psaap-data} hold as equalities at $((h_{M}, \alpha_{M}, v_{M}), M)$.  Put another way, the other 30 data points carry no information about $\widehat{P}$, and could have been ignored.  Also, this finding suggests that the best next experiment to reduce the gap between $\widehat{P}$ and $\P[G(h, \alpha, v) \leq \theta]$ would be to determine $G(h_{M}, \alpha_{M}, v_{M})$, since if it is discovered that in fact $G(h_{M}, \alpha_{M}, v_{M}) \ll M$, then $\widehat{P}$ will decrease considerably.
	
	However, for $\theta \leq 1.0 \, \mathrm{mm}^{2}$, a significant difference ($10^{-2}$ or greater) is observed between the computed $\widehat{P}$ and Markov's bound;  this order-$10^{-2}$ difference was confirmed using runs with an extended convergence criterion (less than \verb!tol!~$=10^{-6}$ change over \verb!ngen!~$=10^{3}$ consecutive generations).  This suggests that data points other than A54 and A67 supply relevant data in these cases, and that it is no longer feasible to have all the $\mu$-probability mass located at $((h_{M}, \alpha_{M}, v_{M}), M)$ and $((h', \alpha', v'), \theta)$.
	
	It is worth noting, though, that working with only the two relevant data points did not result in a statistically significant shortening of the algorithmic run-time.  Instead, significant --- even dramatic --- reductions in computational cost resulted from reducing the dimension of the optimization problem rather than its constraints, as discussed in the next remark.  This is not unexpected:  problem \eqref{eq:opt_prob_reduced} is a problem in $\sim 2^{K}$ unknowns with $\sim K 2^{K} + | \mathcal{O} | 2^{K}$ distinct constraints, so it is unsurprising that $K$ has a much greater effect on computational cost than $| \mathcal{O} |$.
\end{rmk}

\begin{rmk}[Dimensional collapse]
	\label{rmk:dimensional_collapse}
	An interesting empirical observation about the solutions of the optimization problem is that, during the course of the calculation, the approximate maximizers appear to undergo a kind of ``dimensional collapse'', as illustrated in Figure \ref{fig:num_results_collapse}.  That is, the extremizing measure $\mu$ does not have support on the $8$ distinct points of a non-degenerate discrete cube $\mathcal{C}(x_{0}, x_{1})$;  instead, the support of the measure collapses to just one point in the $h$ and $\alpha$ marginals.  This indicates that the uncertainty in the impact velocity $v$ is the dominant uncertainty in this problem.
	
	Furthermore, once this ``dimensional collapse'' phenomenon has been observed, even approximately, it is natural to try the calculation of $\widehat{P}$ using $1 \times 1 \times 2$ product measures instead of $2 \times 2 \times 2$ product measures;  this approach always produces valid lower bounds on $\widehat{P}$ and, as Figure \ref{fig:num_results_cvgce_compare} shows, can greatly reduce the computational burden.  In this way, lower bounds on the solution of a large OUQ problem can be found relatively quickly by considering lower-dimensional sub-problems.  
	
	The automated implementation of this heuristic for general OUQ problems with $n_{1} \times \dots \times n_{K}$ product measures, in which dimensional collapse events are diagnosed ``on the fly'' during an optimization and then enforced as additional simplifying constraints, and the resulting improvements to computational efficiency, will be the topic of a future paper.
\end{rmk}

\begin{figure}
	\includegraphics[width=0.9\linewidth]{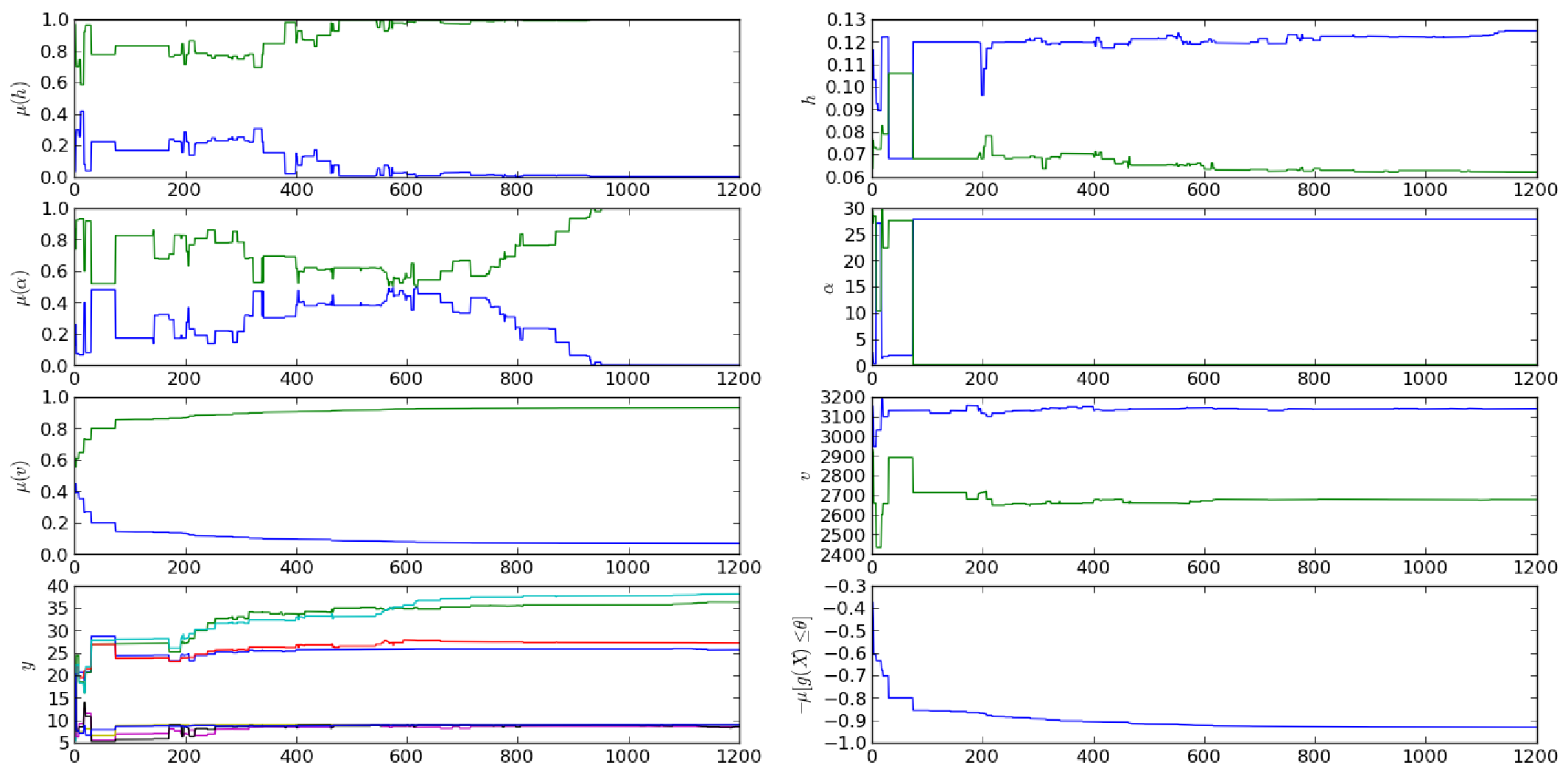}
	\caption{Illustration of the dimensional collapse phenomenon for the approximate maximizers for $\theta = 9.0 \, \text{mm}^{2}$ in Figure \ref{fig:num_results_thetas}.  The first three rows show the $\mu$-probability (left column) and position (right column) of the $h$, $\alpha$ and $v$ coordinates of the support of $\mu$.  The bottom-left figure shows the $y$-values, and the bottom-right the negative of $\mu[g(X) \leq \theta]$, \emph{i.e.}\ $- \widehat{P}_{n}$.  In the later iterations, $\mu$ is effectively a $1 \times 1 \times 2$, not a $2 \times 2 \times 2$, product measure.}
	\label{fig:num_results_collapse}
\end{figure}

\begin{figure}
	\scalebox{0.8}{
		\input{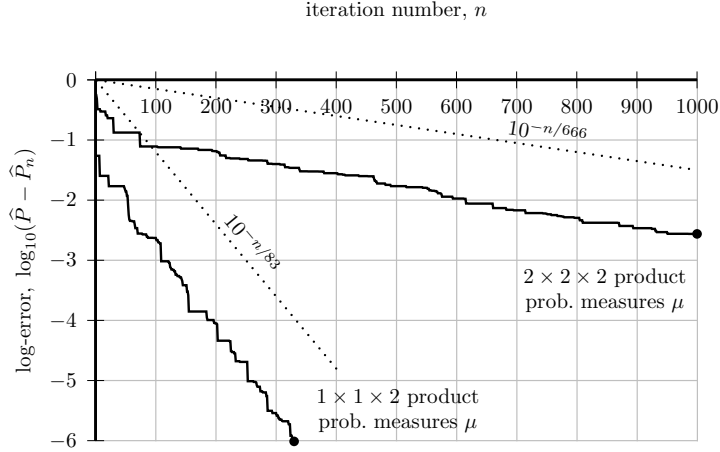}
	}
	\caption{Log-linear plot illustrating typical numerical convergence for the approximate maximum $\widehat{P}_{n}$ for $\theta = 9.0 \, \text{mm}^{2}$ in Figure \ref{fig:num_results_thetas} at full $2 \times 2 \times 2$ dimensionality and reduced $1 \times 1 \times 2$ dimensionality.  Note the improvement to the convergence rate obtained by operating at reduced dimensionality.}
	\label{fig:num_results_cvgce_compare}
\end{figure}

\section{Generalizations}
\label{sec:generalizations}

\subsection{Additional Statistical Information}
\label{subsec:additional_info}

The approach of Section \ref{sec:opt_prob} is open to a great deal of generalization, much more so than that of Section \ref{sec:opt_diam}.  In principle, \emph{any} information about $G$ and $\P$ can be used to define a set of admissible scenarios $\mathcal{A}$ for the optimization problem \eqref{eq:opt_prob}--\eqref{eq:opt_prob_feasible};  also, the objective function can be more general than the probability of failure.  Let $r \colon \mathcal{X} \to \R$ be measurable.  As shown in \cite{OwhadiScovelSullivanMcKernsOrtiz:2010}, if $\mathcal{A}$ is described by independence constraints and inequalities of the form
\begin{align*}
	\E_{\mu} \big[ \varphi'_{i} \big] & \leq 0, && \text{ for $i \in \{ 1, \dots, n' \}$,} \\
	\E_{\mu_{k}} \big[ \varphi^{(k)}_{i} \big] & \leq 0, && \text{ for $k \in \{ 1, \dots, K \}$, $i \in \{ 1, \dots, n_{k} \}$,}
\end{align*}
for given measurable functions $\varphi_{i} \colon \mathcal{X} \to \R$ and $\varphi^{(k)}_{i} \colon \mathcal{X}_{k} \to \R$, then, to extremize $\E_{\mu}[r]$ over $\mu \in \mathcal{A}$, it is sufficient to search over measures $\mu = \bigotimes_{k = 1}^{K} \mu_{k} \in \mathcal{A}$ with $\mu_{k}$ supported on at most $n' + n_{k} + 1$ points of $\mathcal{X}_{k}$;  this paper made use only of the case $r = \one[f \leq \theta]$, $n' = 1$, $\varphi'_{1} = m - f$, $n_{k} \equiv 0$.  In particular, independence assumptions can be relaxed, and information about the moments and correlations of the input random variables $X_{k}$ can be included in the definition of $\mathcal{A}$.  If such information is used, then a reduced upper bound on the probability of failure is obtained, but at the cost of solving a higher-dimensional optimization problem.  

Since, in general, the same methods can be used to provide optimal bounds on $\E_{\mu}[r]$ for any quantity of interest $r$, the methods of this paper can be used to optimally propagate uncertainties through a hierarchy (directed acyclic graph) of partially-observed input-output relationships, as in \cite{TopcuLucasOwhadiOrtiz:2011}.  See Figure \ref{fig:modular} for a schematic illustration.

\begin{figure}
	\scalebox{0.7}{
		\begin{pspicture}(0.0,0.0)(13.0,7.0)
	\psset{linewidth=0.04cm}
	\psset{dotsize=0 6}
	\psset{shadowcolor=lightgray}
	\psset{shadowsize=2pt}
	\psset{shadowangle=-30}
	\psset{arrowsize=0 6}

	\psframe[fillstyle=solid, fillcolor=white, shadow=true](0.0,0.0)(2.5,1.0)
	\rput(1.25,0.5){$\bullet \leq \E[X_{2}] \leq \bullet$}
	
	\psline[shadow=true]{->}(2.5,0.5)(3.5,0.5)
	
	\psframe[fillstyle=solid, fillcolor=white, shadow=true, framearc=0.5](3.5,0.0)(6.0,1.0)
	\rput(4.75,0.5){$\begin{array}{c} \text{$G_{2}$, data $\mathcal{O}_{2}$,} \\ \text{$L_{2}$-Lipschitz} \end{array}$}
	
	\psline[shadow=true]{->}(4.75,1.0)(4.75,2.0)

	\psframe[fillstyle=solid, fillcolor=white, shadow=true](3.5,2.0)(6.0,3.0)
	\rput(4.75,2.5){$\bullet \leq \E[Y_{2}] \leq \bullet$}

	\psframe[fillstyle=solid, fillcolor=white, shadow=true](0.0,6.0)(2.5,7.0)
	\rput(1.25,6.5){$\bullet \leq \E[X_{1}] \leq \bullet$}
	
	\psline[shadow=true]{->}(2.5,6.5)(3.5,6.5)

	\psframe[fillstyle=solid, fillcolor=white, shadow=true, framearc=0.5](3.5,6.0)(6.0,7.0)
	\rput(4.75,6.5){$\begin{array}{c} \text{$G_{1}$, data $\mathcal{O}_{1}$,} \\ \text{$L_{1}$-Lipschitz} \end{array}$}
	
	\psline[shadow=true]{->}(4.75,6.0)(4.75,5.0)

	\psframe[fillstyle=solid, fillcolor=white, shadow=true](3.5,4.0)(6.0,5.0)
	\rput(4.75,4.5){$\bullet \leq \E[Y_{1}] \leq \bullet$}

	\psline[shadow=true]{->}(6.0,4.5)(7.0,3.75)
	\psline[shadow=true]{->}(6.0,2.5)(7.0,3.25)
	
	\psframe[fillstyle=solid, fillcolor=white, shadow=true, framearc=0.5](7.0,3.0)(9.5,4.0)
	\rput(8.25,3.5){$\begin{array}{c} \text{$G_{3}$, data $\mathcal{O}_{3}$,} \\ \text{$L_{3}$-Lipschitz} \end{array}$}
	
	\psline[shadow=true]{->}(9.5,3.5)(10.5,3.5)
	
	\psframe[fillstyle=solid, fillcolor=white, shadow=true](10.5,3.0)(13.0,4.0)
	\rput(11.75,3.5){$\bullet \leq \E[Z] \leq \bullet$}
\end{pspicture}
	}
	\caption{For $i \in \{ 1, 2 \}$, bounds on the expected value of $X_{i}$ can be propagated through a system $G_{i}$ that is known on $\mathcal{O}_{i}$ and has Lipschitz constant $L_{i}$ to yield optimal bounds on the expectation of some output quantity $Y_{i}$.  The bounds on $Y_{1}$ and $Y_{2}$ can then be propagated through a third system $G_{3}$, and so on.}
	\label{fig:modular}
\end{figure}
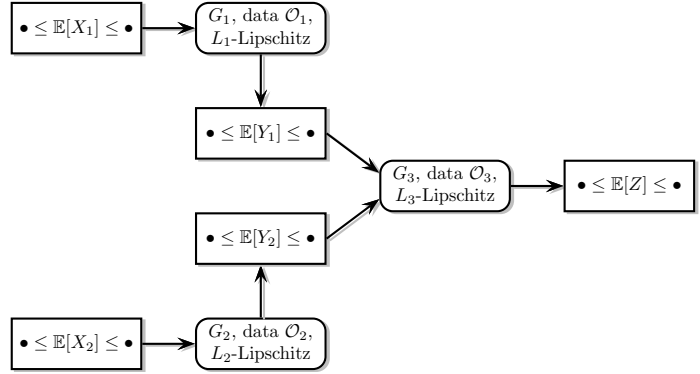

\subsection{Measurement Uncertainty}
\label{subsec:measurement_uncertainty}

Bounded measurement uncertainty can also be incorporated in the inequality constraints.  More precisely, suppose that an error of up to $\pm \delta$ is associated to the observed value $G(z)$, and an error of up to $\delta'$ with respect to the metric $d_{L}$ is associated to the corresponding input parameter value $z$.  Then the observed datum is not $(z, G(z))$ but rather some $\big( \widetilde{z}, \widetilde{G}(\widetilde{z}) \big) \in \mathcal{X} \times \R$ such that
\[
	d_{L}(z, \widetilde{z}) \leq \delta' \text{ and } \left| G(z) - \widetilde{G}(\widetilde{z}) \right| \leq \delta .
\]
In this situation, the Lipschitz constraints of the form
\begin{equation}
	\label{eq:simple_ineq}
	| y - G(z) | \leq d_{L}(x, z)
\end{equation}
generalize to
\begin{equation}
	\label{eq:better_ineq}
	\big| y - \gamma | \leq d_{L}(x, \zeta),
\end{equation}
where $(\zeta, \gamma) \in \mathcal{X} \times \R$ is a new optimization variable that plays the r{\^o}le of the imperfectly-observed input-output pair $(z, G(z))$, and, therefore, is constrained to satisfy
\begin{equation}
	\label{eq:better_ineq_2}
	d_{L}(\zeta, \widetilde{z}) \leq \delta' \text{ and } \left| \gamma - \widetilde{G}(\widetilde{z}) \right| \leq \delta .
\end{equation}
Note that, geometrically, \eqref{eq:better_ineq}--\eqref{eq:better_ineq_2} corresponds to a pointed double cone with a movable vertex that must remain close to $\big( \widetilde{z}, \widetilde{G}(\widetilde{z}) \big)$, whereas \eqref{eq:PSAAP_Lip} corresponds to a fixed and blunt double cone.  Note that, as in the simple situation of Example \ref{eg:Phat_1dim_example}, the bounds $\widehat{D}_{k}$ and $\widehat{P}$ may be discontinuous as functions of $\delta$ and $\delta'$.

If specific statistical information is available about the measurement uncertainty (\emph{e.g.}\ Gaussian scatter), then confidence intervals can be used in the above procedure.  The resulting bounds on $\P[G(X) \leq \theta]$ will be probabilistic in nature, and will become looser as the required level of confidence increases.

\subsection{Model-Based Certification}
\label{subsec:models}

In many applications, although the real response function $G \colon \mathcal{X} \to \R$ cannot be easily exercised, there may be a \emph{model} $F \colon \mathcal{X} \to \R$ for $G$ that can be used instead.  Quantitative relationships between $G$ and $F$ can be used to define sets of admissible scenarios as before.  For example, suppose that it is known that
\begin{equation}
	\label{eq:unif_valid_model}
	\| G - F \|_{\infty} := \sup_{x \in \mathcal{X}} | G(x) - F(x) | \leq C_{V} \text{,}
\end{equation}
where $C_{V} \geq 0$ is some constant resulting from an exercise in \emph{model validation}.  Then, compared with the admissible set $\mathcal{A}$ of \eqref{eq:opt_prob_feasible}, the corresponding set $\mathcal{A}_{F}$ that uses also the model $F$ and the information \eqref{eq:unif_valid_model} is
\[
	\mathcal{A}_{F} := \left\{ (g, \mu) \smid \begin{array}{c}
		g \colon \mathcal{X} \to \R \text{ is $d_{L}$-short,} \\
		\mu = \mu_{1} \otimes \dots \otimes \mu_{K} \in \bigotimes_{k = 1}^{K} \mathcal{P}(\mathcal{X}_{k}) \text{,} \\
		\text{$\| g - F \|_{\infty} \leq C_{V}$, $g = G$ on $\mathcal{O}$, and $\E_{\mu}[g] \geq m$}
	\end{array} \right\} \subseteq \mathcal{A} \text{.}
\]
Hence, in the $\mathcal{A}_{F}$-analogue of the reduced problem \eqref{eq:opt_prob_reduced}, the model $F$ and \eqref{eq:unif_valid_model} induce additional constraints of the form
\[
	| y_{\eps} - F(x_{\eps}) | \leq C_{V} \text{ for each $\eps \in \{ 0, 1 \}^{K}$.} 
\]
As remarked above, $\widehat{D}_{k}$ and $\widehat{P}$ may be discontinuous as functions of $C_{V}$.

Other quantitative measures of model validity can be used in similar ways.  Without going into detail, we note that the uniform norm in \eqref{eq:unif_valid_model} is too strong for many applications, particularly those in which $F$ or $G$ may have discontinuities:  in such cases, $\| F - G \|_{\infty}$ being small requires that $F$ and $G$ have \emph{approximately} the same discontinuities in $\R$ at \emph{exactly} the same locations in $\mathcal{X}$, which is a very strong requirement.  Therefore, metrics that allow ``wiggle room'' in both $\mathcal{X}$ and $\R$, \emph{e.g.}\ the various Skorohod metrics \cite{Billingsley:1999, Skorohod:1956}, are expected to be of use in this area.  For example, it may be reasonable to assume that the distance between the graphs of $F$ and $G$ as subsets of $\mathcal{X} \times \R$ is small enough that, for some $C'_{V} \geq 0$,
\begin{equation}
	\label{eq:Hausdorff_valid_model}
	\sup_{x \in \mathcal{X}} \inf_{x' \in \mathcal{X}} \max \{ d_{L}(x, x'), | G(x) - F(x') | \} \leq C'_{V};
\end{equation}
\emph{i.e.}\ every point on the graph of $G$ lies within distance $C'_{V}$ of some point on the graph of $F$.  (Note well that the roles of $F$ and $G$ in \eqref{eq:Hausdorff_valid_model} are not symmetric.)  In this case, the corresponding constraint satisfied by any feasible $(x_{\eps}, y_{\eps}) \in \mathcal{X} \times \R$ is that
\[
	\inf_{\substack{x' \in \mathcal{X} \\ d_{L}(x_{\eps}, x') \leq C'_{V}}} | y_{\eps} - F(x') | \leq C'_{V}.
\]

\subsection{Set-Valued Lipschitz Functions}
\label{subsec:set-valued}

In many applications (\emph{e.g.}\ inverse problems, which are often ill-posed), the system of interest cannot be accurately represented as a single-valued function $G \colon \mathcal{X} \to \R$.  For example, the system outcome may depend on so-called \emph{unknown unknowns}, which can be neither controlled nor even observed, but have the effect that $G(x)$ is not a uniquely determined real number for each fixed $x \in \mathcal{X}$.  One resolution to this problem is to treat $G$ as a partially-observed \emph{set-valued} function $G \colon \mathcal{X} \rightsquigarrow \R$, \emph{i.e.}\ an operation that assigns to each $x \in \mathcal{X}$ a (possibly empty) subset of $\R$.  There is a notion of Lipschitz continuity for set-valued functions \cite{AubinFrankowska:2009}:  for metric spaces $(\mathcal{X}, d_{\mathcal{X}})$ and $(\mathcal{Y}, d_{\mathcal{Y}})$, a set-valued function $G \colon \mathcal{X} \rightsquigarrow \mathcal{Y}$ is said to be a \emph{set-valued Lipschitz function} with Lipschitz constant $L \geq 0$ if, for all $x, x' \in \mathcal{X}$,
\begin{equation}
	\label{eq:set_valued_Lipschitz}
	G(x) \subseteq \left\{ y \in \R \smid \mathop{\mathrm{dist}} (y, G(x')) := \inf_{y' \in G(x')} d_{\mathcal{Y}}(y, y') \leq L d_{\mathcal{X}}(x, x') \right\},
\end{equation}
that is, $G(x)$ is a subset of the uniform $L d_{\mathcal{X}}(x, x')$-neighbourhood of $G(x')$;  or, equivalently, the Hausdorff distance between the sets $G(x)$ and $G(x')$ is at most $L d_{\mathcal{X}}(x, x')$.  

It would be an interesting and natural extension of the present work to consider set-valued response functions.  Indeed, the set of single-valued Lipschitz extensions $\Extend(\mathcal{X}, G|_{\mathcal{O}}, d_{L})$ as defined in \eqref{eq:extend} defines a set-valued function $\widetilde{G} \colon \mathcal{X} \rightsquigarrow \R$ by
\[
	\widetilde{G}(x) := \{ g(x) \mid x \in \mathcal{X}, g \in \Extend(\mathcal{X}, G|_{\mathcal{O}}, d_{L}) \}.
\]
$\widetilde{G}$ is a set-valued Lipschitz function, with Lipschitz constant $1$ with respect to the metric $d_{L}$, and $\Extend(\mathcal{X}, G|_{\mathcal{O}}, d_{L})$ is the collection of Lipschitz selections \cite[\S9.4.3]{AubinFrankowska:2009} of $\widetilde{G}$.  In this paper, since $G$ is assumed to be single-valued, the sets $\widetilde{G}(x)$ are all convex;  in the general situation, this need not be the case.

\section*{Acknowledgements}

Portions of this work were supported by the US Department of Energy NNSA under award DE-FC52-08NA28613 through the California Institute of Technology's ASC/PSAAP Center for the Predictive Modeling and Simulation of High Energy Density Dynamic Response of Materials.  We thank the California Institute of Technology PSAAP Center's Experimental Science Group --- in particular, M.\ Adams, J.\ M.\ Mihaly and A.\ Rosakis --- for the data set in Table \ref{tbl:psaap-data}.  Finally, we thank three anonymous referees for their helpful comments.

\bibliographystyle{amsplain}
\bibliography{./refs.bib}

\vfill

\end{document}